\documentclass[a4paper,reqno]{amsart}
\usepackage{vmargin}
\usepackage[pdftex]{graphicx}
\usepackage{amsfonts}
\usepackage{amssymb}
\usepackage{amsrefs} 
\usepackage{mathrsfs}
\usepackage{stmaryrd}
\usepackage{amsmath}
\usepackage{amsthm}
\usepackage[shortlabels]{enumitem}
\usepackage{nomencl}
\usepackage[bookmarks=true]{hyperref}   
\usepackage{esint}
\usepackage{dsfont}
\usepackage{upgreek}
\usepackage{xcolor}
\usepackage{slashed}
\usepackage{scalerel}
\usepackage{comment}
\usepackage{todonotes}
\usepackage[utf8]{inputenc}
\usepackage{skull}
\usepackage{tikz-cd}
\usepackage{faktor}
\usepackage{mathtools}
\usepackage{sansmath}
\usepackage{nicefrac}

\hypersetup{
    colorlinks,%
}

\author{Lashi Bandara}
\author{Georges Habib}
\title{Geometric singularities and Hodge theory}
\date{\today}

\address{Lashi Bandara, 
Deakin Mathematics Group, 
Deakin University Melbourne Burwood Campus, 
221 Burwood Highway, Burwood, Victoria, Australia 3125
}
\urladdr{\href{http://www.lashi.org/}{http://www.lashi.org/}}
\email{\href{mailto:lashi.bandara@deakin.edu.au}{lashi.bandara@deakin.edu.au}}

\address{Georges Habib,
Lebanese University,
Faculty of Sciences II,
Department of Mathematics,
P. O. Box 90656 Fanar-Matn,
Lebanon and University of Lorraine, CNRS, IECL, 54506 Nancy, France
}
\urladdr{\href{http://iecl.univ-lorraine.fr/membre-iecl/habib-georges}{https://iecl.univ-lorraine.fr/membre-iecl/habib-georges}}
\email{\href{mailto:ghabib@ul.edu.lb}{ghabib@ul.edu.lb}}


\def\colour{\color}

\def\colour{\colour}

\def\colour{\color}

\newtheorem{theorem}{Theorem}[section]
\newtheorem{corollary}[theorem]{Corollary}
\newtheorem{lemma}[theorem]{Lemma}
\newtheorem{proposition}[theorem]{Proposition}

\newtheorem{definition}[theorem]{Definition}
\newtheorem{remark}[theorem]{Remark}

\newtheorem{example}[theorem]{Example}

\newcommand{\cbrac}[1]{\left(#1\right)}

\newcommand{\dbrac}[1]{\left\{#1\right\}}
\newcommand{\modulus}[1]{|#1|}

\newcommand{\set}[1]{\dbrac{#1}}

\newcommand{\nul}{\mathrm{ker}}
\newcommand{\ran}{\mathrm{ran}}
\newcommand{\dom}{\mathrm{dom}}

\DeclareMathOperator{\ident}{Id}

\newcommand{\comp}{\, \circ\, }

\newcommand{\e}{\mathrm{e}}

\newcommand{\R}{\mathbb{R}}
\newcommand{\Co}{\mathbb{C}}
\newcommand{\C}{\mathbb{C}}

\newcommand{\Ra}{\mathbb{Q}}
\newcommand{\Na}{\ensuremath{\mathbb{N}}}

\newcommand{\script}[1]{\mathscr{#1}}

\renewcommand{\emptyset}{\varnothing}
\newcommand{\union}{\cup}

\newcommand{\intersect}{\cap}

\newcommand{\rest}[1]{{{\lvert_{}}_{}}_{#1}}
\newcommand{\close}[1]{\overline{#1}}		

\newcommand{\powerset}{\script{P}} 

\renewcommand{\epsilon}{\varepsilon}

\renewcommand{\phi}{\varphi}

\newcommand{\graph}{\script{G}}		

\newcommand{\tensor}{\otimes}

\newcommand{\End}{\mathrm{End}} 
\DeclareMathOperator{\Sym}{Sym} 

\newcommand{\norm}[1]{\| #1 \|}			

\newcommand{\spt}[1]{{\rm spt} {\text{ }}#1}	
\DeclareMathOperator{\esssup}{esssup}

\DeclareMathOperator{\divv}{div}		

\DeclareMathOperator{\cut}{\lrcorner}			

\newcommand{\Rcur}{\mathcal{R}}			

\newcommand{\interior}[1]{\mathring{#1}}	
\DeclareSymbolFont{script}{U}{eus}{m}{n}	
\DeclareMathSymbol{\bwedge}{0}{script}{"5E}	
\newcommand{\Forms}[1][{}]{\bwedge^{#1} {\kern 0.05em}}		
\newcommand{\Tensors}[1][{}]{{\mathcal{T}}^{(#1)}}	
\newcommand{\tanb}{{\rm T}}		
\newcommand{\cotanb}{{\rm T}^\ast}	

\newcommand{\restrict}{\scalebox{1.7}{\,\ensuremath{\llcorner}\,}}  

\newcommand{\extd}{{\rm d}}			

\newcommand{\inprod}[1]{\left\langle #1 \right\rangle}	
\newcommand{\Leb}[1][{}]{\script{L}^{#1}}			

\DeclareMathOperator{\cliff}{{\scaleobj{0.5}{\triangle}}}	

\newcommand{\sym}{\upsigma}
\newcommand{\spec}{\mathrm{spec}}				

\newcommand{\conj}[1]{\overline{#1}}				

\newcommand{\Lp}[2][{}]{{\rm L}^{#2}_{\rm #1}}		
\newcommand{\Ck}[2][{}]{{\rm C}^{#2}_{\rm #1}}		
\newcommand{\Hard}[2][{}]{{\rm H}^{#2}_{\rm #1}}		
\newcommand{\SobH}[2][{}]{\Hard[#1]{\rm #2}}

\newcommand{\cB}{\mathcal{B}}

\newcommand{\cE}{\mathcal{E}}

\newcommand{\cM}{\mathcal{M}}

\newcommand{\mg}{\mathrm{g}}
\newcommand{\mgt}{\tilde{\mg}}
\newcommand{\mh}{\mathrm{h}}
\newcommand{\mht}{\tilde{\mh}}
\newcommand{\mut}{\tilde{\mu}}

\newcommand{\odd}{\mathrm{odd}}

\newcommand{\Dir}{{\rm D} }
\newcommand{\DDir}{{\mathbb{D}} }

\newcommand{\Lap}{\Delta}

\newcommand{\dM}{\partial \cM}

\newcommand{\met}{\mathrm{dist}}

\newcommand{\Sph}{\mathrm{S}}

\renewcommand{\imath}{\mathrm{i}}		

\newcommand{\Hom}[1][*]{\mathcal{H}^{#1}}
\newcommand{\Proj}[1]{\mathcal{P}_{#1}}

\newcommand{\Meas}{\mathrm{Meas}}
\newcommand{\Zero}{\mathrm{Zero}} 

\newcommand{\eextd}{\mathbf{d}}
\newcommand{\Hodge}{\mathrm{Hodge}}
\newcommand{\dR}{\mathrm{dR}}

\newcommand{\cc}{\mathrm{cc}}
\newcommand{\dd}{\mathrm{dd}}

\newcommand{\Spa}{\mathcal{X}}
\newcommand{\Hil}{\script{H}}			
\DeclareMathOperator{\sgn}{sgn}
\newcommand{\eoplus}{\oplus_{\mathrm{ext}}}

\newcommand{\Sing}{\mathrm{sing}}

\newcommand{\Mat}{\mathrm{Mat}}

\begin{document}

\maketitle
\begin{abstract}
We consider smooth vector bundles over smooth manifolds equipped with non-smooth geometric data.
For nilpotent differential operators acting on these bundles, we show that the kernels of induced Hodge-Dirac-type operators remain isomorphic under uniform perturbations of the geometric data.
We consider applications of this to the Hodge-Dirac operator on differential forms induced by so-called rough Riemannian metrics, which can be of only measurable coefficient in regularity, on both compact and non-compact settings.
As a consequence, we show that the kernel of the associated non-smooth Hodge-Dirac operator with respect to a rough Riemannian metric remains isomorphic to smooth and singular cohomology when the underlying manifold is compact. 
\end{abstract} 
\tableofcontents

\parindent0cm
\setlength{\parskip}{\baselineskip}

\section{Introduction}

Hodge theory, which introduces differential forms to the study of topology, initiated by  Élie Cartan and Georges de Rham, is a classic and well studied topic.   
In modern language, on an oriented smooth closed manifold $\cM$  of dimension $n$, de Rham showed in his thesis \cite{deRham} that there is a canonical isomorphism  between the singular cohomology $\Hom[*]_{\Sing}(\cM)$ and the de Rham cohomology $\Hom[*]_{\dR}(\cM)$.
That is, for each $k\in \{0,\ldots,\dim\cM\}$, we have 
\begin{equation} 
\label{Eq:CanIso} 
\Hom[k]_{\Sing}(\cM) \cong \Hom[k]_{\dR}(\cM) := \faktor{\ker (\extd\rest{\Ck{\infty}(\cM;\Forms[k]\cM)})}{ \ran (\extd\rest{\Ck{\infty}(\cM;\Forms[k+1] \cM)})}, 
\end{equation} 
where  $\Forms[k]\cM \to \cM$  is the bundle of $k$-forms on $\cM$ and $\extd: \Ck{\infty}(\cM; \Forms[k]\cM) \to \Ck{\infty}(\cM; \Forms[k+1]\cM)$ is the exterior derivative.
In other words, de Rham's theorem asserts that the singular cohomology of $\cM$ can  be computed via solutions to $\extd \omega = 0$.

An important analytic perspective is to recover a similar isomorphism from the point of view of elliptic operators.
Namely, if we fix a smooth Riemannian metric $\mg$ on $\cM$, it induces a canonical metric on $\Forms\cM = \oplus_{k=0}^n \Forms[k]\cM$. 
Furthermore,  if we denote by  $\extd^{\dagger,\mg}$ the formal adjoint of $\extd$ via $\mg$, the operator $\DDir_{\mg,\Hodge} = \extd + \extd^{\dagger,\mg}$ called the \emph{Hodge-Dirac} operator, is an important formally self-adjoint  first-order \emph{elliptic}  differential operator. 
Its square, $\DDir_{\mg,\Hodge}^2$, called the \emph{Hodge-Laplacian}, is often the focus of Hodge theory. 
When $\cM$ is assumed to be closed, the Hodge-Dirac operator  admits a unique  self-adjoint  extension $\close{\DDir_{\mg,\Hodge}}$ in $\Lp{2}(\cM;\Forms\cM)$, where this space is induced by the metric $\mg$.
An $\Lp{2}$  version of the Hodge-theorem asserts that 
\begin{equation} 
\label{Eq:Hodge} 
\Hom[k]_{\dR}(\cM) \cong \ker( \close{\DDir_{\mg,\Hodge}}\rest{\Lp{2}(\cM;\Forms[k]\cM)}) = \ker( \close{\DDir_{\mg,\Hodge}}^2 \rest{\Lp{2}(\cM;\Forms[k]\cM)}).
\end{equation}
From this, if $\mh$ is another smooth metric, we see that
\begin{equation} 
\label{Eq:KerIso}  
\ker( \close{\DDir_{\mg,\Hodge}}) \cong  \ker( \close{\DDir_{\mh,\Hodge}}).
\end{equation}
The isomorphism \eqref{Eq:Hodge} is of enormous significance as it  enables uses of elliptic operator methods  to the study of topology.

Inspired by  the isomorphism \eqref{Eq:Hodge} and \eqref{Eq:KerIso},  we formulate the following two questions central to this paper:  
\begin{enumerate}[label=(Q\arabic*)]
\item \label{Q1} 
Does there exist an isomorphism \eqref{Eq:Hodge} in the situation where $\mg$ is no longer smooth? 
\item \label{Q2} 
Does there exist an isomorphism \eqref{Eq:KerIso}  beyond the closed or even compact  manifolds  setting, again for $\mg$ and $\mh$ non- smooth?  
\end{enumerate}

Related questions to \ref{Q1} and \ref{Q2} have already been considered. 
For instance, a slight variation related to \ref{Q1} is pursued by Teleman in \cite{Teleman}.
There, the manifold $\cM$ is assumed to merely Lipschitz - i.e., the differential structure consists of local lipeomorphisms - although it is further assumed to be closed.
The lack of regularity in the differential structure complicates matters as $\Forms[k]\cM \to \cM$ is now a vector bundle with only measurable-coefficient transition maps.
In particular, there is no smooth cohomology $\Hom[k]_{\dR}(\cM)$. 
Therefore, the author instead phrases the question in terms of the $\Lp{2}$-cohomology given by 
\[ 
\Hom[k](\close{\extd}) :=  \faktor{\ker (\close{\extd}\rest{\Lp{2}(\cM;\Forms[k]\cM)})}{\ran (\close{\extd}\rest{\Lp{2}(\cM;\Forms[k+1]\cM)})}.
\]
There, it is shown that $\Hom[k](\close{\extd}) \cong  \ker (\close{\DDir_{\mg,\Hodge}}\rest{\Lp{2}(\cM;\Forms[k]\cM)})$. 
If now $\cM$ was further assumed to be smooth,  then by the well known fact $\Hom[k](\close{\extd}) \cong \Hom[k]_{\dR}(\cM)$,  we obtain a positive answer to \ref{Q1}.

However, our approach follows a different line of inquiry, guided by a desire to consider both \ref{Q1} and \ref{Q2} through a unified perspective.
To keep matters simple, we restrict ourselves $\cM$ possessing a smooth differential structure, although the results we obtain in this paper, appropriately phrased, could be obtained in the Lipschitz setting.
Note that if we obtained a positive answer to  \ref{Q2} and further restrict ourselves to closed $\cM$, we can expect \ref{Q1} to follow. 
With this in mind, we promote \ref{Q2} to be the central question we investigate through identifying a sufficiently general setup so that \ref{Q1} would follow.
We emphasise we only expect an isomorphism and allow for the possibility that the kernels are infinite dimensional.

More explicitly, the setup we consider is to fix $\cM$ an arbitrary smooth manifold, which may even have boundary. 
We fix a smooth vector bundle $\cE \to \cM$ supporting a \emph{rough differential operator} $\Dir: \Ck{\infty}(\cM;\cE) \to \Lp{0}(\cM;\cE)$ (where $\Lp{0}(\cM;\cE)$ is the set of measurable sections).
The geometric data comes from considering a  \emph{locally elliptic compatible measure} $\mu$ (c.f. Definition~\ref{Def:LECM}), which is a Radon measure locally bounded below and which agrees with the intrinsic measure structure of $\cM$. 
Furthermore, we fix $\mh^{\cE}$ to be a \emph{rough Hermitian metric}, which we only ask to be measurable coefficient and locally bounded above and below in an appropriate sense (c.f. Definition~\ref{Def:RHM}).
Under these conditions, we obtain a canonical space $\Lp{2}(\cM;\cE, \mh^{\cE},\mu)$, which captures the geometric data.
In this space, we consider closed extensions $\Dir_e$ of  $\Dir$, that are \emph{nilpotent}, by which we mean that $\ran(\Dir_{e}) \subset \ker(\Dir_{e})$.
As this operator is densely-defined, a consequence of standard operator theory is that there exists a densely-defined adjoint $\Dir_{e}^{\ast,\mu, \mh^{\cE}}$.
The nilpotency assumption is non-trivial,  even in the case that $\ran(\Dir) \subset \Ck{\infty}(\cM;\cE)$ (i.e., $\Dir$ is smooth coefficient) and $\Dir^2 = 0$.
That is, it does not follow that an extension $\Dir_{e}$ is automatically nilpotent.
We highlight this by example in Subsection~\ref{Sec:Nonnil}.

Unlike in the smooth situation, we might find $\Ck[cc]{\infty}(\cM;\cE) \not\subset \dom(\Dir_{e}^{\ast,\mu,\mh^{\cE}})$ (where $\Ck[cc]{\infty}(\cM;\cE)$ the subset of compactly supported smooth sections supported in the interior), despite $\Ck[cc]{\infty}(\cM;\cE) \subset \dom(\Dir_{e})$.
We provide an example of this in Subsection~\ref{Sec:NonsmoothDom}.
Mimicking the smooth metric situation, we need to ensure that $\DDir_{e,\mg} := \Dir_{e} + \Dir_{e}^{\ast,\mu,\mh^{\cE}}$ yields a closed and densely-defined operator so that it can be studied via analytic methods.
In general, the sum of two densely-defined and closed operators may not be densely-defined and closed.
This is a non  issue in the smooth case as $\Ck[cc]{\infty}(\cM;\cE) \subset \dom(\Dir_{e}^{\ast,\mu,\mh^{\cE}})$ and so the operator $\DDir_{e,\mg}$ is automatically densely-defined. 
In our situation, more subtle arguments are needed.

To that end, and motivated by other reasons which will become clear in this paper,   we utilise Hodge-theory for closed, densely-defined and nilpotent operators $\Gamma: \dom(\Gamma) \subset \Hil \to \Hil$ on a Hilbert space $\Hil$ as developed by Axelsson (Rosén)-Keith-McIntosh in \cite{AKMc}.
In this setting, we obtain a Hilbert Hodge decomposition 
\begin{equation} 
\label{Eq:HodgeDecomp}
\Hil = \ker(\Gamma) \cap \ker(\Gamma^\ast) \oplus \close{\ran(\Gamma)} \oplus \close{\ran(\Gamma^\ast)}.
\end{equation}
This decomposition is a special case of Hilbert space Hodge-theory which the authors in \cite{AKMc} establish. 
Their primary purpose is to use this theory to obtain a first-order perspective for the Kato square root problem for divergence form operators, which effectively requires perturbed adjoints resulting in Hodge-Dirac-type operators which are bi-sectorial and in general non-self-adjoint.
This is in an even more general setup than we require in this paper, so we provide an exposition of this theory for the self-adjoint version through explicit, detailed and conceptual proofs in Appendix~\ref{Appendix} of this paper. 
We note that much of the structure theory for Hodge-decompositions in Hilbert spaces and, more generally, in Banach spaces, have been established and utilised by many authors. 
See, for instance \cite{AxMc, GMM, Mon21, Bai20, McMon18, FMcP18, SL16, BMc16, HMcP11,AMcR, HMcP08, Mor12, B13}, although we emphasise that this is by no means an exhaustive list.

The decomposition \eqref{Eq:HodgeDecomp} ensures that the generalised Hodge-Dirac-type operator $\Pi := \Gamma + \Gamma^{\ast}$ is self-adjoint, which in particular means that it is also densely-defined and closed.
As aforementioned, it is generally untrue that the sum of two densely-defined operators will also be densely-defined, and so we provide an explicit example in Subsection~\ref{Ex:HodgeMinDomain} for the exterior derivative with boundary conditions to illustrate why we can expect the sum domain to be dense in the general setting.
Applying this theory to $\Gamma = \Dir_{e}$ then allows us to obtain $\DDir_{e,\mu,\mh^{\cE}}$ as a self-adjoint operator.
Fixing another measure and metric $\mut$ and $\mht^{\cE}$ that are uniformly comparable almost-everywhere to $\mu$ and $\mh^{\cE}$ in an ``$\Lp{\infty}$'' sense from above and below (c.f. Definition~\ref{Def:UniComp}), the first part of our main result in Theorem~\ref{Thm:GenHodgeMain} asserts that
\begin{equation}
\label{Eq:AbsKerIso} 
\ker (\DDir_{e,\mu,\mh^{\cE}}) \cong \ker(\DDir_{e,\mut,\mht^{\cE}}).
\end{equation}
This is a general form of \ref{Q2} and choosing $\cE = \Forms\cM$ and $\Dir = \extd $, we obtain an affirmative answer to Question~\ref{Q2} for \emph{rough Riemannian metrics}  $\mg$ and $\mh$ (measurable-coefficient metrics locally bounded above and below; c.f. Definition~\ref{Def:RRM}), when $\mg$ and $\mh$ are mutually uniformly bounded.

The second part of our main result in Theorem~\ref{Thm:GenHodgeMain} concerns restrictions of this isomorphisms for subspaces when the bundle exhibits a splitting of the form $\cE = \cE_0 \oplus \cE_1$. 
Under commutativity of the projection along this with some power of the operator $\DDir_{e,\mu,\mh^{\cE}}$ and $\DDir_{e,\mut,\mht^{\cE}}$ , we obtain a more refined isomorphism
\[
\ker(\DDir_{e,\mu,\mh^{\cE}}\rest{\cE_i})  \cong  \ker(\DDir_{e,\mut,\mht^{\cE}}\rest{\cE_i}).
\] 
Consequently, again choosing $\cE = \Forms\cM$ and $\Dir = \extd$, when  $\cM$ is closed and  $\mg$ a rough Riemannian metric, we obtain an affirmative answer to \ref{Q1}.
This also justifies and supplants credence to the terminology ``geometric singularity'' which we have used in the  title of this article, as it shows that the underlying cohomology captured by the operators with non-smooth coefficients remain isomorphic to the smooth cohomology of the manifold when it is compact.
In other words, while these operators capture the underlying geometry and fail elliptic regularity, they are structurally unaffected by the singularities present in these metrics. 

Let us briefly describe the conceptual abstract picture behind the isomorphism \eqref{Eq:AbsKerIso}.
Given $B: \Hil \to \Hil$ a bounded, invertible, positive-definite, symmetric operator, we can induce a comparable inner product $(u,v) \mapsto \inprod{Bu,v}$.
Letting $\Gamma^{\ast,B}$ be the adjoint of $\Gamma$ computed with respect to  this  new inner product and $\Pi_B:=\Gamma + \Gamma^{\ast, B}$, we prove that
\[ 
\Hil = \ker(\Pi_B) \oplus \close{\ran(\Gamma)} \oplus \close{\ran(\Gamma^\ast)}.
\]
Here, the sum is only topological, not necessarily orthogonal.
Note that $\ker(\Gamma) \cap \ker(\Gamma^{\ast}) = \ker(\Pi)$, and  from \eqref{Eq:HodgeDecomp}, 
\[
 \ker(\Pi_B) \oplus \close{\ran(\Gamma)} \oplus \close{\ran(\Gamma^\ast)} = \ker(\Pi) \oplus \close{\ran(\Gamma)} \oplus \close{\ran(\Gamma^\ast)}
\] 
from which the isomorphism 
\[ 
\ker(\Pi_B) \cong \ker(\Pi) 
\]
readily follows.
Conceptually speaking, we can see this decomposition as essentially saying that Hodge-Dirac-type operators arising from a change of metric simply sees a ``rotation'' of its kernel within $\Hil$, which results in the isomorphism which we seek. 

It should be noted that the Hodge-decomposition \eqref{Eq:HodgeDecomp} applied to $\Gamma$ obtained with metrics $(u,v) \mapsto \inprod{u,v}$ as well as $(u,v) \mapsto \inprod{Bu,v}$ provides us an alternative isomorphism through
\[
\ker(\Gamma) = \ker(\Pi) \oplus \close{\ran(\Gamma)} = \ker(\Pi_B) \oplus \close{\ran(\Gamma)}. 
\]
However, this is purely algebraic and while an important consequence of the Hodge-decomposition, it evades our desire to introduce operator methods. 
The perspective we develop  illustrates that isomorphisms of kernels are a consequence of the nilpotent structure, whereas the  identification of the kernel itself  typically requires analytic methods.
To this end, it is vital that we obtain a well defined Hodge-Dirac-type operator.
This division of labour is seen best when the metric is smooth and the manifold is compact.
There, we see the nilpotency of the exterior derivative allows yields a $\Lp{2}$-Hodge-theory along with an isomorphism to the smooth cohomology, while the identification of the kernel as having smooth sections arise from elliptic regularity theory for the Hodge-Dirac operator.

One motivation we have in writing this paper is to promote the development of these techniques beyond the non-degenerate setting that we have explored here.
In this paper, we ultimately show that the singularities we study are not exhibited topologically in the sense that isomorphisms of the Hodge-Dirac operator are retained.
However, it would be interesting to know whether these methods could be adapted to study situations where the metric may degenerate and where this degeneration may be captured by operator methods.
This would provide a novel approach to compute topological information in the presence of more severe singularities, those beyond that are ``geometric singularities''.
Such singularities arise in many other important contexts, such as in algebraic geometry and during collapse at maximal time in a geometric flow. 
In particular, there is a desire to obtain Betti numbers for collapsed spaces along a Ricci flow, and this was a question that prompted the results we present here in this paper. 

On a related note, we wish to emphasise that the order of differentiation of our operators plays no role in the isomorphism we obtain.
Although this approach is considered a ``first-order'' approach as it was formulated in the context of the first-order characterisation of the Kato square root problem, it is really the nilpotent structure which lies at the crux of our considerations.
We anticipate this being important in applications where the order of differentiation may change, which occurs in some situations outside of the Riemannian context.
Since we still take an operator perspective, we expect that this will still allow for analytic methods to identify the kernels of operators in question.

This paper is structured as follows.
In Section~\ref{Sec:SetupResults}, we showcase our general setup, taking care to precisely defining the measures, metrics and  their properties.  
The main theorem we prove in this paper is Theorem~\ref{Thm:GenHodgeMain} is presented here.
In Section~\ref{Sec:Examples}, we provide many examples, including applications to Hodge-Dirac operators for non-smooth metrics in Subsections~\ref{Ex:HodgeDirac} and \ref{Ex:HodgeLaplace}, applications to non-compact manifolds in Subsection~\ref{Ex:Noncompact}, Weakly Lipschitz boundary in Subsection~\ref{Sec:WLB}, considerations of relative and absolute cohomology in Subsection~\ref{Sec:RelAbs}, results pertaining to Hodge-magnets as Hodge-Dirac-Schrödinger operators in Subsection~\ref{Sec:Hodge-magnet} and bundle-valued forms and flat connections in Subsection~\ref{Sec:BVF}. 
We also compute particular instances  to highlight the techniques  we use via the way of examples in this section.
In Section~\ref{Sec:GH}, we provide detailed proofs of the necessary properties of the objects involved in the statement of Theorem~\ref{Thm:GenHodgeMain} which we presented in Section~\ref{Sec:SetupResults}. 
Moreover, in this section, we provide the proof of Theorem~\ref{Thm:GenHodgeMain}. 
We also present results for first-order operators here, to make it computationally straightforward  to verify the hypothesis of Theorem~\ref{Thm:GenHodgeMain} for large classes of operators which arise naturally in geometry. 
In Section~\ref{Sec:RRMHodge}, we consider applications to Hodge-Dirac operators arising from rough metrics.
In this latter section, we also consider conditions under which we are guaranteed nilpotent extensions.
As aforementioned, Appendix~\ref{Appendix} contains the abstract results that we use. 
Some of the results here are already known in the literature and this is highlighted as the appendix develops. 

\section*{Acknowledgements}
This research was initiated  at Brunel University London where L.B. was located at the time. 
We thank the Department of Mathematics at BUL for hosting G.H and their hospitality.
G.H. would like also to thank the Atiyah-UK Lebanon fellowship for the financial support.

We also thank Veronique Fischer, Francesca Tripaldi and Louis Yudowitz for useful discussions that have contributed to improve this paper.

\section{Setup and results}
\label{Sec:SetupResults}
\subsection{Notation} 

Throughout, we use the analysts inequality $x \lesssim y$ to mean that there is some constant $C$ for which $x \leq C y$.
The dependency of the constant $C$ will be clear either from context or it will be explicitly stated. 
By $x \simeq y$, we mean that $x \lesssim y$ and $y \lesssim x$.

We also use Einstein summation convention - when a raised index is multiplicatively adjacent to a lowered index, we implicitly sum over that index.

Throughout, we let $\cM$ be a smooth manifold (with or without boundary), by which we mean that the manifold has a smooth differential structure.
The tensor bundle of covariant rank $p$ and contravariant rank $q$ will be denoted by $\Tensors[p,q] \cM$ and the tangent and cotangent bundles by $\tanb\cM$ and $\cotanb\cM$. 
The bundle of $k$ differential forms will be denoted by $\Forms[k]\cM$ and the wedge product is denoted by $\wedge$.
 
When $\cM$ is equipped with a metric $\mg$, there is a canonical extension to a metric on $\Forms[k]\cM$. 
The wedge project then has an adjoint, the \emph{cut} product denoted by $\cut_{\mg}$, defined as $\mg(\alpha \wedge \beta, \gamma) = \mg(\beta, \alpha \cut_{\mg} \gamma)$.
The Clifford product is denoted by $\cliff_{\mg} = \wedge + \cut_{\mg}$.

For a vector bundle $\cE \to \cM$, we denote the $k$-differential sections by $\Ck{k}(\cM;\cE)$. 
Compactly supported sections (up to the boundary if $\dM \neq \emptyset$) is the subspace $\Ck[c]{k}(\cM;\cE)$ and compactly supported sections with support in the interior (i.e. away from the boundary) will be denoted by $\Ck[cc]{k}(\cM;\cE)$. 
Note that when $\dM = \emptyset$, we have $\Ck[c]{k}(\cM;\cE) = \Ck[cc]{k}(\cM;\cE)$.
The support of a section $\xi: \cM \to \cE$ will be denoted by $\spt \xi$. 

By \emph{measure}, we shall always take it to mean an outer measure (so that it is defined on the entire powerset $\powerset(\cM)$) and take the measure $\sigma$-algebra as those sets satisfying the Carathéodory criterion.
We denote the measure $\sigma$-algebra of $\mu$ by $\Meas(\mu)$ and the measure zero subalgebra by $\Zero(\mu)$.
Given $A \subset \cM$ which is $\mu$-measurable, the restriction measure to $A$ will be denoted by $(\mu \restrict A)(B) = \mu(A \cap B)$.
If $\mu$ and $\nu$ are Radon measures, then Radon-Nikodym derivative of $\mu$ with respect to $\nu$ will be denoted by $\frac{d\mu}{d\nu}$.

\subsection{Measurability, zero measure sets and non-smooth metrics}

The locally Euclidean structure of $\cM$ not only affords us with a topology, but also an  \emph{intrinsic} measure structure.
We begin with the following definition, noting that $\Leb$ is the Lebesgue measure on $\R^n$.
\begin{definition}[Measurable and zero measure]
\label{Def:MeasZero}
Define $\Meas \subset \powerset(\cM)$ and $\Zero \subset \powerset(\cM)$ by:
\begin{align*} 
\Meas &:= \set{A \in \powerset(\cM): \forall (\psi,U) \text{ charts,}\  \psi(A \cap U) \subset \R^n \text{ is } \Leb-\text{measurable}}, \\
\Zero &:= \set{A \in \powerset(\cM): \forall (\psi,U) \text{ charts,}\ \Leb (\psi(A \cap U)) = 0}.
\end{align*}
\end{definition}

In fact $\Zero  \subset \Meas$ and are, in fact, $\sigma$-algebras (c.f.  Proposition~\ref{Prop:MeasZero}). 
Therefore, by calling  $\Meas$ the \emph{(intrinsic) measurable sets} of $\cM$ and  $\Zero$ the \emph{(intrinsic) zero measure sets} of $\cM$, we obtain an  \emph{intrinsic} notion of measurability for $\cM$. 
In particular,  this allows us to talk about measurable sections of a bundle $\cE \to \cM$, as well as the property of an expression being valid almost-everywhere, without having to fix a reference measure.
More precisely, this ensures the following notion is well defined.

\begin{definition}[Measurable section]
\label{Def:MeasSec}  
Let $\cE \to \cM$ and suppose that $\xi:\cM \to \cE$ is a section, by which we mean a map (not necessarily smooth) satisfying $\xi(x) \in \cE_x$.
If inside every local trivialisation, the coefficients of $\xi$ are measurable, we call $\xi$ a measurable section.
We denote the set of measurable sections by $\Lp{0}(\cM;\cE)$.
\end{definition}

We now move onto  considering  more general classes of measures which are compatible with this underlying and intrinsic measure structure. 

\begin{definition}[(Locally elliptic) compatible measures]
\label{Def:CompMeas}
\label{Def:LECM}
A measure $\mu$ is said to be a \emph{$\cM$-compatible measure} or simply a \emph{compatible measure} if  it is a Radon measure with $\Meas(\mu) = \Meas$ and $\Zero(\mu) = \Zero$. If in addition there exists a cover $\set{(U_\alpha, \psi_\alpha)}$ by charts with constants $C_\alpha \geq 1$ such that the Radon-Nikodym derivative of $\mu$ with respect to the pullback Lebesgue measure $\psi_\alpha^\ast (\Leb \restrict \psi_\alpha(U_\alpha))$ in $U_\alpha$ satisfies: 
$$ \frac{1}{C_\alpha} \leq \frac{d(\mu\restrict U_{\alpha})}{d\psi_\alpha^\ast (\Leb \restrict \psi_\alpha(U_\alpha))}(x) \leq C_\alpha$$
for $x$-a.e. in $U_\alpha$, then we call $\mu$ a \emph{locally elliptic compatible measure} or \emph{LECM} for short.
\end{definition}

There are many measures which satisfy this criterion.
For instance, Proposition~\ref{Prop:MeasZero} shows that the induced volume measure $\mu_{\mg}$ associated to a smooth metric $\mg$ is an LECM.
Also weighted measures of the form $\mu_f := f \mu_{\mg}$ where $f \in \Ck{\infty}(\cM;[0, \infty))$ for a smooth $\mg$ preserves the LECM property. 

In the setting of a smooth manifold which we want to equip with non-smooth geometric data,  it is of vital importance that the measure algebra and zero sets measures we consider in Definition~\ref{Def:CompMeas} are equal to the sets $\Meas$ and $\Zero$. 
In what is to follow, we will be looking at non-smooth objects which require these notions of measurability and set of measure zero, in order to formulate a condition holding almost-everywhere.
It will be apparent later that this sense of ``compatibility'' guarantees function spaces and other related objects are well defined.

\begin{definition}[Rough Hermitian metric (RHM)]
\label{Def:RHM} 
Let $\cE \to \cM$ be a vector bundle and let $\mh \in \Lp{0}(\cM; \Sym_{\C} \cE^\ast \tensor \cE^\ast)$. 
We say that $\mh$ is a rough Hermitian metric if there exists a cover for $\cM$ by open sets $U_{\alpha}$  for which  there is a constant $C_\alpha \geq 1$, and smooth Hermitian metric $\mh^{\infty}_{\alpha}$ on $U_{\alpha}$ with 
$$ \frac{1}{C_\alpha} \modulus{u}_{\mh^{\infty}_{\alpha}(x)} \leq \modulus{u}_{\mh(x)} \leq C_\alpha \modulus{u}_{\mh^{\infty}_{\alpha}(x)}$$
for almost every $x \in U_\alpha$ and for all $u \in \cE_x$.
A pair $(U_\alpha,\mh^{\infty}_{\alpha})$ is called a \emph{local comparability structure}.
\end{definition}

It is readily seen that this definition captures all possible smooth Hermitian metrics $\mh$ on a given bundle $\cE \to \cM$ by taking that smooth metric itself as $\mh^\infty_\alpha$ and with the single open set $U_\alpha = \cM$.
If instead $\mh$ was continuous, then on choosing a cover of $\cM$ by precompact open sets, it is clear that any smooth metric on the closure will satisfy the inequality in the definition.
Therefore, this definition also includes all continuous Hermitian metrics on $\cE \to \cM$.

Note that although the bundle $\cE \to \cM$ above can be a real bundle, in which case $\Sym_{\C} = \Sym_{\R}$ meaning real-symmetric, we will almost exclusively be considering complex bundles.
When the bundle is determined by the geometry, for instance the bundle of forms, we will first complexify it in the canonical manner to consider it as a complex bundle.
That being said, for real bundles, we can consider rough Riemannian/real metrics (RRMs).
Later we will see that RRMs are important as non-smooth Riemannian metrics on $\cM$ itself.

Throughout, our primary considerations will involve pairs $(\mu,\mh)$ where $\mh$ is a RHM and $\mu$ a LECM. 
By definition, it is clear that a RHM is $\mu$-measurable in the sense of the LECM, and therefore, we can integrate against these objects to construct Lebesgue spaces by defining
$$\Lp{p}(\cM;\cE, \mu,\mh) := \set{ u \in \Lp{0}(\cM;\cE): \int_{\cM} \modulus{u}^{p}_{\mh}\ d\mu < \infty}$$
when $1 \leq p < \infty$, and for $p = \infty$,
$$ \Lp{\infty}(\cM;\cE,\mu,\mh) := \set{ u \in \Lp{0}(\cM;\cE): \esssup \modulus{u}_{\mh} < \infty }.$$
Note that $\Lp{\infty}(\cM;\cE,\mu,\mh)$ is independent of $\mu$ since all compatible measures share the same sets of measure and measure zero, and therefore, we often simply write it as $\Lp{\infty}(\cM;\cE,\mh)$.
The norms in each case are obvious. 

Moreover, by $\Lp{p}(\cM) = \Lp{\infty}(\cM;\C)$, we mean $p$-integrable functions  $\xi: \cM \to \C$ where the RHM here is implicitly  $\modulus{\cdot}$, the usual absolute value in $\C$.
As before, we emphasise that $\Lp{\infty}(\cM)$ is independent of measure since by definition, the notion of almost-everywhere agree between any two LECMs.

\subsection{Nilpotent differential operators and their behaviour under mutual boundedness}
Since we deal with non-smooth metrics and measures in this paper, we also allow our fundamental operators to possess non-smooth coefficients. 
We begin with the following definition. 

\begin{definition}[Rough differential operator]
\label{Def:RDO}
An operator $\Dir: \Ck{\infty}(\cM;\cE) \to \Lp{0}(\cM;\cE)$ is called a \emph{rough differential operator} or \emph{RDO} if it is a \emph{local} operator (i.e., $\spt \Dir f \subset \spt f$). Given $(\mu,\mh)$ an LECM-RHM pair, if $\Dir_{e}: \dom(\Dir_e) \subset \Lp{2}(\cM;\cE, \mu,\mh) \to \Lp{2}(\cM;\cE,\mu,\mh)$ with $\Ck[cc]{\infty}(\cM;\cE) \subset \dom(\Dir_e)$ and $\Dir_e\rest{\Ck[cc]{\infty}(\cM;\cE)} = \Dir\rest{\Ck[cc]{\infty}(\cM;\cE)}$, then we say that $\Dir_e$ is an extension of $D$.
\end{definition}

Note the \emph{rough} aspect comes here from the fact that we do not demand $\ran(\Dir) \subset \Ck{\infty}(\cM;\cE)$.
This should not be interpreted as some notion of infinite differentiation.
In typical applications, this could arise from adding a non-smooth potential to a smooth coefficient differential operator.

Our forthcoming results will be phrased for closed extensions $\Dir_e$ associated to a rough differential operator $\Dir$ in $\Lp{2}(\cM;\cE, \mu,\mh)$ with respect to a LECM-RHM pair $(\mu,\mh)$.
Certainly, the existence of closed extensions is well known, at least in specific contexts. 

For instance, if $\Dir$ is smooth coefficient, then $\Dir_{\cc}:=\Dir\rest{\Ck[cc]{\infty}(\cM;\cE)}$  is closable. 
However, more generally, this is more subtle and sufficient conditions for existence of closed extensions beyond the smooth context will be treated later. 
A notion that is fundamental to our results is the following. 

\begin{definition}
\label{Def:Nilpotent}
Let $\Dir_{e}$ be an extension (not necessarily closed) of  $\Dir$. 
Then $\Dir_{e}$ is said to be \emph{nilpotent} if $\ran(\Dir_e) \subset \ker(\Dir_e)$.
\end{definition}

Note that if the operator $\Dir_{e}$ is indeed closed, then $\ker(\Dir_e)$ is closed and so nilpotency  self-improves to $\close{\ran(\Dir_e)} \subset \ker(\Dir_e)$.

In reality, when the operator has smooth coefficients, the operators we meet will often satisfy $\Dir^2 = 0$ on $\Ck{\infty}(\cM;\cE)$. 
Under very mild conditions, this would yield that $\Dir_{e}$ is also nilpotent.
However, this is not always the case, as highlighted in the example given in Subsection~\ref{Sec:Nonnil}. 
That being said, in Proposition~\ref{Prop:NilExt}  we provide sufficient conditions to ensure the nilpotency of $\Dir_{e}$.

In what is to follow, we will compare the kernels of operators obtained from distinct LECM-RHM pairs.
For this, and in order to mimic ``$\Lp{\infty}$'' without an underlying reference metric, we define the following.

\begin{definition}[Mutually bounded measures, RHMS and pairs]
\label{Def:UniComp}
\label{Def:UniRHM}
\label{Def:MutBdd}
\begin{enumerate}
\item 
Two LECMs $\mu_1$ and $\mu_2$ are said to be \emph{mutually bounded} if the Radon-Nikodym derivatives satisfy
$$ \frac{d \mu_1}{d \mu_2} \in \Lp{\infty}(\cM)\quad\text{and}\quad \frac{d\mu_2}{d\mu_1} \in \Lp{\infty}(\cM).$$
In this case, we write $\mu_1 \sim \mu_2$. 
\item 
Let $\cE \to \cM$ be a vector bundle and $\mh_1$ and $\mh_2$ be two RHMs for which there exists $C \geq 1$ satisfying 
$$ \frac{1}{C} \modulus{u}_{\mh_1(x)} \leq \modulus{u}_{\mh_2(x)} \leq C \modulus{u}_{\mh_1(x)}.$$
Then, the two metrics are said to be \emph{mutually bounded} with constant $C$ and we succinctly denote this by $\mh_1 \sim \mh_2$.
\item 
If $(\mu_1, \mh_1)$ and $(\mu_2,\mh_2)$ are two LECM-RHM pairs and $\mu_1 \sim \mu_2$ and $\mh_1 \sim \mh_2$, then we say that they are \emph{mutually bounded} and write $(\mu_1, \mh_1) \sim (\mu_2,\mh_2)$.
\end{enumerate} 
\end{definition} 

An important property of mutually bounded pairs $(\mu_1, \mh_1)$ and $(\mu_2,\mh_2)$ is that  $\Lp{2}(\cM;\cE,\mu_1, \mh_1) = \Lp{2}(\cM;\cE,\mu_1, \mh_1)$ as sets with equivalent norms  $\norm{\cdot}_{\Lp{2}(\cM;\cE,\mu_1, \mh_1)} \simeq \norm{\cdot}_{\Lp{2}(\cM;\cE,\mu_2, \mh_2)}$. 
The constant appearing here are related to the constants appearing in Definition~\ref{Def:MutBdd}.
See Proposition~\ref{Prop:LpFix} for a detailed description.
With this in hand, we now present the central result of this paper.

\begin{theorem}[Generalised Hodge-type theorem]
\label{Thm:GenHodgeMain}
For two mutually bounded pairs $(\mu_1,\mh_1)$ and $(\mu_2,\mh_2)$, we have $\Lp{2}(\cM;\cE) := \Lp{2}(\cM;\cE,\mu_1, \mh_1) = \Lp{2}(\cM;\cE,\mu_2, \mh_2)$ with equality as sets with equivalence of norms.
An extension $\Dir_e$  of a  rough differential  operator $\Dir$ is closed and nilpotent with respect to $\norm{\cdot}_{(\mu_1, \mh_1)}$ if and only if it is closed and nilpotent with respect to $\norm{\cdot}_{(\mu_2, \mh_2)}$. 
For such a $\Dir_e$, defining $\DDir_{e,i} := \Dir_e + \Dir_e^{\ast, (\mu_i,\mh_i)}$, where $\Dir_e^{\ast,(\mu_i,\mh_i)}$ is the adjoint of $\Dir_e$ with respect to $\inprod{\cdot,\cdot}_{(\mu_i,\mh_i)}$, the following hold.
\begin{enumerate}[label=(\roman*)]
\item  
\label{Itm:GenHodgeMain:1} 
The operator $\DDir_{e,i}$ is self-adjoint with respect to $\inprod{\cdot,\cdot}_{(\mu_i, \mh_i)}$ and in particular, it is densely-defined and closed. 
More generally, the operators $\DDir_{e,i}$ are bi-sectorial with respect to either norm.
\item  
\label{Itm:GenHodgeMain:2} 
For all $k,l,k',l' \in \Na\setminus\set{0}$, the equalities  $\ker(\DDir_{e,i}^k) = \ker(\DDir_{e,i})$ and $\close{\ran(\DDir_{e,i}^l)} = \close{\ran(\DDir_{e,i})}$ hold, and moreover,
\[
\Lp{2}(\cM;\cE) 
= 
\ker(\DDir_{e,1}^k) \oplus \close{\ran(\DDir_{e,1}^l)} 
= 
\ker(\DDir_{e,2}^{k'}) \oplus \close{\ran(\DDir_{e,1}^{l'})},
\]
where the splitting is topological but not in general orthogonal.
\item 
\label{Itm:GenHodgeMain:3} 
We have that  $\ker(\DDir_{e,1}) \cong \ker(\DDir_{e,2})$ where the isomorphism is given by 
$$ \Phi :=  \Proj{\ker(\DDir_{e,1}), \close{\ran(\DDir_{e,1})}}\rest{\ker(\DDir_{e,2})}: \ker(\DDir_{e,2}) \to \ker(\DDir_{e,1}).$$
The inverse is then given by
$$ \Phi^{-1} :=  \Proj{\ker(\DDir_{e,2}), \close{\ran(\DDir_{e,1})}}\rest{\ker(\DDir_{e,1})}: \ker(\DDir_{e,1}) \to \ker(\DDir_{e,2}).$$
\item 
\label{Itm:GenHodgeMain:4} 
If $\cE = \bigoplus_{j=1}^K \cE_j$ and there exists some power $k \in \Na$  such that $[\DDir_{e,i}^k, \Proj{\cE_j, \oplus_{l\neq j}^K \cE_l}] = 0$ for $i=1,2$ and for all $j = 1, \dots, N$,
then $\ker(\DDir_{e,1}\rest{\cE_j}) \cong \ker(\DDir_{e,2}\rest{\cE_i})$, where the isomorphisms are given by
\[ 
\Phi_j 
:=  
\Proj{\ker(\DDir_{e,1}), \close{\ran(\DDir_{e,1})}}\rest{\ker(\DDir_{e,2}\rest{\cE_j})}: \ker(\DDir_{e,2}\rest{\cE_j}) \to \ker(\DDir_{e,1}\rest{\cE_j}).
\]
\end{enumerate}
\end{theorem}

\begin{remark} 
As observed in Remark~\ref{Rmk:AbsHodgeType}, an alternative map for the isomorphism in \ref{Itm:GenHodgeMain:3} is given by 
$\Proj{\ker(\DDir_{e,1}), \close{\ran(\Dir_{e})}}\rest{\ker(\DDir_{e,2})}: \ker(\DDir_{e,2}) \to \ker(\DDir_{e,1})$.  Note that this isomorphism is different to that in  Theorem~\ref{Thm:GenHodgeMain}.
There, the projection is along $\close{\ran(\DDir_{e})}$, whereas here, the projection is along $\close{\ran(\Dir_{e})}$. 
\end{remark}  

\begin{remark} 
We cannot call the operator $\DDir_{e,i}$ a rough differential operator as we have defined this notion.
In fact, even when the operator $\Dir_{e}$ has smooth coefficients, a lack of regularity of $(\mu_i, \mh_i)$  may cause $\Ck[cc]{\infty}(\cM;\cE) \not\subset \dom(\Dir_{e}^{\ast, (\mu_i, \mh_i)})$.
See subsection~\ref{Sec:NonsmoothDom} for an example. 
\end{remark}

\section{Examples}
\label{Sec:Examples}

\subsection{Rough metrics in the wild}
In this subsection, we will show naturally occurring examples of RRMs, RHMs and LECMs from both analysis and geometry.

\subsubsection{Geometry of divergence-form operators}

We recall divergence-form operators with measurable coefficient operators, popularised by the a priori estimates obtained by De Giorgi-Moser-Nash \cite{dG,Moser,Nash}.
These operators can be seen from a geometric perspective to be within the RHM-LECM framework.

Let $A \in \Lp{\infty}(\R^n; \Sym \Mat(\R^n))$ such that $A^{-1} \in \Lp{\infty}(\R^n; \Sym \Mat(\R^n))$, where $\Sym \Mat(\R^n)$ denotes symmetric matrices on $\R^n$.
The divergence-form operator is often written as 
\[
L_{A} := -\divv A \nabla: \dom(\L_{A}) \subset \SobH{1}(\R^n) \subset \Lp{2}(\R^n) \to \Lp{2}(\R^n)
\]  
which we prefer to write equivalently as  $L_{A} = \extd^\ast A \extd$.
Letting $u \in \dom(L_{A})$ and $v \in \SobH{1}(\R^n)$,
\[
\inprod{L_{A} u, v}_{\Lp{2}} 
= 
\int_{\R^n} (\extd^\ast A \extd u)\ \conj{v}\ d\Leb
=
\int_{\R^n} (A \extd u) \cdot \conj{\extd v}\ d\Leb
= 
\inprod{A \extd u, \extd v}_{\Lp{2}}.
\]

Define $\mg_{A}(x) [u,v] := A(x) u \cdot {v}$. 
Clearly, since $A, A^{-1} \in\Lp{\infty}(\R^n; \Sym \Mat(\R^n))$, $\mg_{A}$ is an RRM  on $\R^n$.
Trivially, the Lebesgue measure $\Leb$ is an LECM. 
Complexifying $\cotanb \R^n$ and extending $\mg_{A}$ naturally, we see that  $(\mg_{A}, \Leb)$ is an RHM-LECM pair on the smooth manifold $\R^n$.
Moreover,  $L_{A} = \Delta_{\mg_A,\Leb}$, the Laplacian on $(\R^n, \mg_{A}, \Leb)$.

More generally, let $\cM$ be a smooth manifold with a smooth metric $\mg$.
The Sobolev space $\SobH{1}(\cM)$ is obtained as the maximal domain of $\extd$ with respect to $\mg$.
Let $A, A^{-1} \in \Lp{\infty}(\Sym \End (\cotanb \cM))$. 
Then, the operator $L_{A, \mg} := \extd^{\ast,\mg} A \extd$ is self-adjoint obtained by the energy $\cE_{A,\mg}[u,v] = \inprod{A \extd u, \extd v}_{\Lp{2}(\cotanb\cM)}$ on $\SobH{1}(\cM)$.
As before, writing $\mg_{A}(x)[u,v] := \mg(x) [ A(x)u,v]$, we obtain that $L_{A} = \Lap_{\mg_A, \mu_{\mg}}$ is the Laplacian on  $(\cM, \mg_{A}, \mu_{\mg})$, with respect to the RHM-LECM pair $(\mg_{A},\mu_{\mg})$.

\subsubsection{Lipschitz graphs}

Let $f: \R^n \to \R$ be a locally Lipschitz map.
Then, the graphical map  $\Phi := x \mapsto (x, f(x)): \R^n \to \graph (f)$ generates a Lipschitz hypersurface.
The inverse of $\Phi$ is $\Phi^{-1} = \graph(f) \ni (x,y) \mapsto x$. 
It is immediate that $\Phi$ is a Lipeomorphism.
 
For $u,v \in \tanb_x \R^n$, let 
\[ \mg_f(x) [u,v] 
:= 
(\Phi^{\ast}(x) \inprod{\cdot,\cdot}_{\R^{n+1}})[u,v] 
= \inprod{ \extd \Phi(x) u, \extd \Phi (x)  v}_{\R^{n+1}}.
\]
By construction,  $(\R^n, \mg_f)$ is isometric to $(\graph(f), \inprod{\cdot,\cdot}\rest{\graph(f)})$.  

It is easy to see that $\mg_f$ is a RRM and in general, its coefficients will be at most measurable. 
In fact, the singularities of $f$ can be severe, which can be seen in the case  $n=2$. 
There, given a set of measure zero $Z \subset \R^2$, there exist a locally Lipschitz map $f_Z: \R^2 \to \R$ such that $Z$ is contained in the singular set of $f_Z$.
By choosing $Z = \Ra \times \Ra$, a dense subset of $\R^2$, we obtain that $ \graph(f_Z)$, and hence $(\R^2, \mg_{f_Z})$, contains a dense set of singular points.
Nevertheless, $\mg_{f_Z}$ is still a RRM, and the results of this paper apply. 

This process can also be replicated on a smooth Riemannian manifold $(\cM,\mg)$. 
First consider $f: \cM \to \R$ a locally Lipschitz map, which induces $\cM \ni x \mapsto (x,f(x)) \in \cM \times \R$.
Since $\graph(f) \subset \cM \times \R$, the natural metric to consider  is  $\mh := \mg + dt^2$ on $\cM \times \R$. 
Then, let $\mg_f := \Phi^\ast \mh$, which is an  RRM on $\cM$.

\subsubsection{Finite dimensional Alexandrov spaces of curvature bounded below}

Let $(\Spa,\met)$ be a connected geodesic metric space. 
In particular, the intrinsic metric, given by the Hausdorff $1$-measure, agrees with the metric $\met$. 
Moreover, any two points can be connected by a curve of shortest length.
Alexandrov spaces are a subclass of such metric spaces which carry a notion of  bounded \emph{synthetic sectional curvature}, a definition we will recall.

Let $a, b \in \Spa$ and write $\modulus{ab} = \met(a,b)$.  
Denote a curve between these two points with length $\modulus{ab}$ by $[ab]$.
For three points $a, b, c \in \Spa$, we  let $\triangle abc$ be a triangle given by sides of length  $[ab]$, $[bc]$ and $[ac]$ of lengths $|ab|,\ |bc|,\ |ac|$ respectively.
Note that the points $a,b,c$ do not uniquely define triangles $\triangle abc$. 

In order to definite a curvature lower bound, we need to compare $\triangle abc$ to triangles $\triangle ABC$  in Euclidean space, hyperbolic space, or the sphere. 
To that end, for $k \in \R$, define the \emph{$k$-space} as $\R^2$ for $k = 0$, the sphere of radius $1/ \sqrt{k}$ for $k > 0$ and  hyperbolic space of radius $1/\sqrt{-k}$ for $k < 0$.
By  $\triangle ABC$, we denote a triangle  in $k$-space with the same side lengths as $\triangle abc$.
That is $|ab| = |AB|,\ |bc| = |BC|,\ |ac| = |AC|$. 

If $k \in \R$ and every point $x \in \Spa$ has an open neighbourhood $U_x$ such that for any triangle $\triangle abc$ contained in $U_x$ and any point $d \in [ac]$ satisfies $|bd| \geq |BD|$ for $\triangle ABC$ and $D \in [AC]$ in the $k$-space, then $(\Spa,\met)$ is said to be an Alexandrov space with (sectional) curvature $\geq k$.

Alexandrov spaces with curvature $\geq k$ arise naturally. 
If $k \in \R$ and $D > 0$ and  $(\cM_j, \mg_j)$ are a sequence of $n$-dimensional smooth Riemannian manifolds with diameter $\leq D$ and sectional curvature $\geq k$, then there exists an Alexandrov space $(\Spa_\infty,\met_\infty)$ with curvature $\geq k$, diameter $\leq D$, and Hausdorff dimension $\leq n$  and a  Gromov-Hausdorff convergent subsequence $(\cM_{j_k}, \mg_{j_k})$ which converges to $(\Spa_\infty,\met_\infty)$.

A general Alexandrov space $(\Spa,\met)$ of curvature $\geq k$ for $k \in \R$ is either infinite dimensional or finite and integer dimensional (measured in terms of the Hausdorff dimension). 
To proceed, we further  assume that $(\Spa,\met)$ is an $n$-dimensional   Alexandrov space of curvature bounded below by some $\kappa \in \R$.
However, we do not assume that its diameter is bounded.

Otsu and Shioya in \cite{OS} define a notion of singular points $S_{\Spa}$ associated with $\Spa$. 
In Theorem A, they show that the Hausdorff dimension of this set of singular points is $\leq (n-1)$. 
Moreover, in Theorem B, they show that $\cM := \Spa \setminus S_{\Spa}$ is a $\Ck{1}$-manifold carrying a $\Ck{0}$ Riemannian metric $\mg$, which induces $\met \rest{\cM}$.

Recall that by considering the maximal atlas, there exists a compatible $\Ck{\infty}$ chart for a $\Ck{1}$-manifold. 
Passing to such a chart, $\cM = \Spa \setminus S_{\Spa}$ is a smooth manifold.
Therefore $\mg$ is a RRM on a smooth manifold $\cM$. 
Consequently, the results obtained in this paper apply to Hodge-Dirac operators constructed out of the exterior derivative $\extd$ on $(\cM,\mg)$.
In particular, it is interesting to know whether an appropriate nilpotent extension $\extd_{e}$ can be chosen for $\cM$ such that $\Dir_{e} = \extd_{e} + \extd_{e}^\ast$ encodes information about the entire space $\Spa$.

\subsection{The Hodge-Dirac and Hodge-Laplace operators on compact manifolds without boundary}

Throughout, let $\cM$ be a compact manifold without boundary. 
Fix two rough metrics $\mg$ and $\mh$ (c.f. Definition~\ref{Def:RRM}).

\subsubsection{Hodge-Dirac}
\label{Ex:HodgeDirac}
To consider the Hodge-Dirac operator, denote the exterior derivative by  $\extd$. 
Then, we have
$$\ker(\extd + \extd^{\ast,\mg}) \cong \ker(\extd + \extd^{\ast,\mh}) \cong \Hom[\ast]_{\dR}(\cM),$$
where the second isomorphism is obtained via using an auxiliary smooth metric and invoking the classic Hodge theorem for that metric.
Moreover, 
\[
\ker((\extd + \extd^{\ast,\mg})\rest{\Forms[k]\cM}) 
\cong 
\ker((\extd + \extd^{\ast,\mh})\rest{\Forms[k]\cM}) 
\cong 
\Hom[k]_{\dR}(\cM)
\]
for $0 \leq k \leq n$, where $\Hom[k]_{\dR}(\cM)$ is the $k$-th de-Rham cohomology of $\cM$.
This is proved in Section~\ref{Sec:RRMHodge} as Theorem~\ref{Thm:HodgeDiracThm}.

\subsubsection{Hodge-Laplace}
\label{Ex:HodgeLaplace}
We also gain information regarding the Hodge-Laplacian, which is the typical object studied in the literature. 

Letting $\Lap_{\mg,\Hodge} = (\extd + \extd^{\ast,\mg})^2 = \extd \extd^{\ast,\mg} + \extd^{\ast,\mg} \extd$, we obtain the Hodge theorem 
\[ 
\ker(\Lap_{\mg,\Hodge}) \cong \ker(\Lap_{\mh,\Hodge}) \cong \Hom[\ast]_{\dR}(\cM)
\]
as usually presented in the literature.
Moreover, 
\[ 
\ker(\Lap_{\mg,\Hodge}\rest{\Forms[k]\cM}) \cong \ker( \Lap_{\mh,\Hodge}\rest{\Forms[k]\cM}) \cong \Hom[k]_{\dR}(\cM)
\]
for $0 \leq k \leq n$.
This simply follows from the fact that $\ker(T) = \ker(T^2)$ for a self-adjoint operator $T$ and by Example~\ref{Ex:HodgeDirac}.

\subsubsection{A perturbed Hodge-Dirac}
Let us also consider the operator $\eextd := \imath \extd$, where $\imath = \sqrt{-1}$ and $\extd$ is the exterior derivative.
Clearly, $\eextd$ is nilpotent and  $\eextd^{\ast,\mg} = - \imath \extd^{\ast,\mg}$.
Define $\DDir_{\mg} := \eextd + \eextd^{\ast,\mg} = \imath\DDir_{\mg,\mathrm{diff}}$, where $\DDir_{\mg,\mathrm{diff}} := \extd - \extd^{\ast,\mg}$.

Given another rough metric $\mh$, Theorem~\ref{Thm:GenHodgeMain} yields that $\ker(\DDir_{\mg}) \cong \ker(\DDir_{\mh})$. 
As with the Hodge-Dirac operator, it is readily verified that $\DDir_{\mg}^2$ commutes with the projection to the bundle of $k$-forms, and therefore, 
$\ker(\DDir_{\mg}\rest{\Forms[k]}) \cong \ker(\DDir_{\mh}\rest{\Forms[k]})$.

Note that 
$$\ker(\DDir_{\mg}) = \ker(\imath \DDir_{\mg,\mathrm{diff}}) = \ker(\extd) \cap \ker(\extd^{\ast,\mg}) = \ker(\DDir_{\mg,\Hodge}).$$
Therefore,
from Example~\ref{Ex:HodgeDirac}, we obtain that 
$$ \ker(\DDir_{\mg,\mathrm{diff}}) \cong \Hom[\ast]_{\dR}(\cM)\quad\text{and}\quad \ker(\DDir_{\mg,\mathrm{diff}}\rest{\Forms[k]\cM}) \cong \Hom[k]_{\dR}(\cM).$$
Note that this conclusion could alternatively be obtained on observing that $\DDir_{\mg,\mathrm{diff}}^2 = - \DDir_{\mg,\Hodge}^2 = - \Lap_{\mg,\Hodge}$.

\subsection{Non-compactness}
\label{Ex:Noncompact}

Let us consider a generalisation from the previous set of examples to the non-compact setting.
In the approach we have taken, we can simply work with non-compact manifolds in much the same way, taking care to ensure that we are now working on an extension of the exterior derivative.

More precisely, we let $\cM$ be a smooth manifold without boundary and $\mg$ and $\mh$ two RRMs on $\cM$ that are mutually bounded.
That is, there is a constant $C \geq 1$ such that $C^{-1} \modulus{u}_{\mg} \leq \modulus{u}_{\mh} \leq C\modulus{u}_{\mg}$, see Definition~\ref{Def:UniComp}.
Clearly, $\met_{\mg}$ is complete if and only if $\met_{\mh}$ is complete. 

When $\mg$ is smooth and $\met_{\mg}$ is complete, $\extd$ on $\Ck[cc]{\infty}(\cM;\Forms\cM)$ extends uniquely to an operator on $\Lp{2}(\cM;\Forms\cM)$.
However, if $\met_{\mg}$ fails to be complete, even for smooth $\mg$, we may find that $\extd$  on $\Ck[cc]{\infty}(\cM;\Forms\cM)$ fails to have a unique and canonical extension in $\Lp{2}(\cM;\Forms\cM)$.
Therefore, for an RRM $\mg$, we proceed generally by considering two canonical extensions $\extd_{0} = \close{\extd_{\cc}}$ where $\dom(\extd_{\cc}) := \Ck[cc]{\infty}(\cM; \Forms \cM)$ and $\extd_{2} = \close{\extd_{\dd}}$ with $\dom(\extd_{\dd}) = \set{u \in \Ck{\infty} \cap \Lp{2}(\cM; \Forms\cM): \extd u \in \Lp{2}(\cM; \Forms\cM)}$. 
These are shown to be well defined in Subsection~\ref{Sec:SmoothCoeff}.

Then, by Corollary~\ref{Cor:HodgeDiracThmMinMaxExt}, we obtain that
\begin{align*} 
\ker(\extd_{0} + \extd_{0}^{\ast,\mg}) \cong \ker(\extd_{0} + \extd_{0}^{\ast,\mh}) \quad\text{and}\quad
\ker(\extd_{2} + \extd_{2}^{\ast,\mg}) \cong \ker(\extd_{2} + \extd_{2}^{\ast,\mh}).
\end{align*}

These results also hold if the rough Riemannian metrics $\mg$ and $\mh$ are instead replaced with RHM-LECM pairs $(\tilde{\mg},\mu_1)$ and $(\tilde{\mh}, \mu_2)$ that are mutually bounded.

\subsection{Weakly Lipschitz boundary}
\label{Sec:WLB} 

Let $\cM$  be a manifold with smooth interior but with weakly Lipschitz boundary $\partial \cM  \neq \emptyset$.
We assume that $\mg$ is an RRM on $\cM$. 

\label{Ex:HodgeBoundary}
The operators $\extd_{0} = \close{\extd_{\cc}}$ where $\dom(\extd_{\cc}) := \Ck[cc]{\infty}(\cM; \Forms \cM)$ and $\extd_{2} = \close{\extd_{\dd}}$ with $\dom(\extd_{\dd}) = \set{u \in \Ck{\infty} \cap \Lp{2}(\cM; \Forms\cM): \extd u \in \Lp{2}(\cM; \Forms\cM)}$  are defined  in Subsection~\ref{Sec:SmoothCoeff}. 
Then, by Corollary~\ref{Cor:HodgeDiracThmMinMaxExt}, we obtain that
\begin{align*} 
\ker(\extd_{0} + \extd_{0}^{\ast,\mg}) \cong \ker(\extd_{0} + \extd_{0}^{\ast,\mh}) \quad\text{and}\quad 
\ker(\extd_{2} + \extd_{2}^{\ast,\mg}) \cong \ker(\extd_{2} + \extd_{2}^{\ast,\mh})
\end{align*}
for another RRM $\mh$ which is mutually bounded with respect to $\mg$. 
In particular, if $\cM = \Omega \subset \R^n$, a compactly contained weakly Lipschitz domain, then $\mh$ can be chosen as $\mh = \inprod{\cdot,\cdot}_{\R^n}$, the Euclidean inner product.
In this case, aspects of results in \cite{AxMc} are recovered, although the authors do not consider the metric perturbation question there. 

Note that the weakly Lipschitz boundary case is a special case of the example in Subsection~\ref{Ex:Noncompact}.
As the problem is phrased purely in terms of the operator in the interior, we can instead consider $\interior{\cM}$ and carry out the analysis there.
When there is a boundary that is weakly Lipschitz, it is possible to identify the domains of these interior operators via the boundary restriction map.

\subsection{Relative and absolute boundary conditions and cohomology}
\label{Sec:RelAbs}

Let $(\cM,\mg)$ be a smooth Riemannian manifold with compact boundary.
We fix $\vec{n}$ to be the normal vector to $\partial \cM$ and $\tau$ the associated covector.
As in Example~7.25 in \cite{BB12}, we see that for $x \in \dM$,
\[ 
\Forms[j]\cotanb_x \cM = (\Forms[j] \cotanb_x \partial \cM) \oplus \tau(x) \wedge (\Forms[j-1] \cotanb_x \partial \cM). 
\]
Therefore, $\phi \in \Lp{0}(\partial \cM; \Forms\cotanb \cM)$ decomposes as $\phi = \phi_{t} + \tau \wedge \phi_{n}$.
The absolute and relative boundary conditions for the Hodge-Dirac operator $\extd + \extd^{\dagger}$ are then given by the subspaces
\[ 
B_{\mathrm{abs}} := \set{ \phi \in \SobH{\frac12}(\partial \cM;\Forms\cM): \phi_{n} = 0}\ \ \text{and}\ \ 
B_{\mathrm{rel}} := \set{ \phi \in \SobH{\frac12}(\partial \cM;\Forms\cM): \phi_{t} = 0}.
\]
These are elliptically regular local boundary conditions as demonstrated in Example~7.25 in \cite{BB12}. 

We remark that when $\phi_{n} = 0$, only the ``tangential'' component $\phi_{t}$ remains.
Similarly, when $\phi_{t} = 0$, only the ``normal'' component $\phi_{n}$ remains.
Furthermore, note that 
 \[
\tau \wedge \phi = \tau\wedge \phi_{t} \quad\text{and}\quad \tau \lrcorner \phi = \phi_n. 
\] 
Therefore, $\phi_n = 0$ if and only if  $\tau \lrcorner \phi = 0$  and $\phi_t = 0$ if and only if  $\tau \wedge \phi = 0$. 

This is the view taken in \cite{AxMc}.
Their attention is confined to bounded domains $\Omega \subset \R^n$, albeit they also permit the boundary to be of lower regularity.
For $F \in \Ck{\infty}(\Omega; \Forms\Omega)$, they consider a distribution $F_z$ by extending $F$ by zero to all of $\R^n$ outside of $\Omega$. 
Then, inside $\Omega$
\[ 
\extd F_z = (\extd F)\rest{\Omega} - (\tau \wedge f)\sigma \quad\text{and}\quad
\extd^{\dagger} F_z = \extd^{\dagger}F\rest{\Omega} - (\tau \cut f)\sigma, 
\] 
where $f$ is the restriction to the boundary and $\sigma$ is the surface measure on $\partial \Omega$.
Nevertheless, they observe that $\extd F_z \in \Lp{2}(\R^n; \Forms\R^n)$ if and only if $\tau \wedge f = 0$ and similarly for $\extd^{\dagger} F_z$.
They define ``$\extd$ with normal boundary conditions'' as $\extd_{\close{\Omega}}$ specified by $ \dom ( \extd_{\bar{\Omega}}) = \set{F \in \Lp{2}(\Omega;\Forms\Omega): \extd F_z \in \Lp{2}(\R^n;\Forms\R^n)}$  and ``$\extd^{\dagger}$ with tangential boundary conditions'' as $\extd^{\dagger}_{\close{\Omega}}$ with $ \dom ( \extd^\dagger_{\bar{\Omega}}) = \set{F \in \Lp{2}(\Omega;\Forms\Omega): \extd^\dagger F_z \in \Lp{2}(\R^n;\Forms\R^n)}$.
Moreover, they consider  the maximal extensions and we write them as  $\extd_{\Omega,\max}$ and $\extd^{\dagger}_{\Omega,\max}$ since the metric in this case is Euclidean and smooth.

Furthermore, they show  that $\extd_{\close{\Omega}}^\ast = \extd^{\dagger}_{\Omega,\max}$ and $(\extd^{\dagger}_{\close{\Omega}})^\ast = \extd_{\Omega,\max}$.
Assuming the boundary $\partial \Omega$ is smooth, we apply Example~7.25 in \cite{BB12} to find that the operator  $\Dir_{n} := \extd_{\close{\Omega}} + \extd^{\dagger}_{\Omega,\max}$ precisely corresponds to relative boundary conditions and $\Dir_{t} := \extd^\dagger_{\close{\Omega}} + \extd_{\Omega,\max}$ corresponds to absolute boundary conditions.

Clearly, these calculations hold if $\Omega$ is considered to be a general  Riemannian manifold $(\cM,\mg)$ with compact boundary.
However, this is a moot point for our exposition. 
Here, what is essential is the structure of the operators corresponding to relative and absolute boundary conditions as given above. 
This is key to characterising this to our setting.

To continue, let $\cM$ be a manifold with compact boundary with RRM $\mg$. 
The operator $\extd_{0}$ as defined in Subsection~\ref{Sec:WLB} is then captured by $\extd_{\close{\Omega}}$. 
Similarly, the operator $\extd_{2}$ as defined in Subsection~\ref{Sec:WLB} then precisely captures the operator $\extd_{\Omega,\max}$. 
This is a consequence of Proposition~4.3 in \cite{AxMc}, where they show that $\Ck[cc]{\infty}(\Omega;\Forms\Omega)$ is dense in $\dom(\extd_{\close{\Omega}})$ and where the second equality further requires the observation that $\Ck[cc]{\infty}(\Omega;\Forms\Omega)$ is dense in $\dom(\extd^{\dagger}_{\close{\Omega}})$ and using the fact that $\extd_{\Omega,\max} = (\extd^{\dagger}_{\close{\Omega}})^\ast$.
Consistently generalising the picture we have discussed, we define  $\Dir_{\mathrm{rel},\mg} := \extd_{0} + \extd_{0}^{\ast,\mg}$ which captures the relative boundary conditions and $\Dir_{\mathrm{nor},\mg} := \extd_{2} + \extd_{2}^{\ast,\mg}$.

In this situation, for another RRM $\mh$ such that $\mh \sim \mg$, we again have by Corollary~\ref{Cor:HodgeDiracThmMinMaxExt}
\begin{align*} 
\ker(\Dir_{\mathrm{rel},\mg}) \cong \ker(\Dir_{\mathrm{rel},\mh})\quad\text{and}\quad 
\ker(\Dir_{\mathrm{abs},\mg}) \cong \ker(\Dir_{\mathrm{abs},\mh}). 
\end{align*}
When the manifold $\cM$ is compact, $\mh$ can be chosen to be smooth which provides a tether to the classical results pertaining to relative and absolute cohomology. 

\subsection{Generalised Hodge-magnets}
\label{Sec:Hodge-magnet}

Let us now fix $\cM$ a smooth manifold,  a RHM $\mg$ on $\Forms\cM$ and a LECM $\mu$. 
Fix $\extd_e$ to be a closed nilpotent extension of $\extd_0$ with respect to the data $(\mg, \mu)$.
Furthermore, fix $\omega \in \dom(\extd_e) \cap \Lp{0}(\cM; \Forms[k]\cM)$.
Then, consider the operator $\Dir_{e}$ given by 
\[
\Dir_e u := \extd_e u + \omega \wedge u.
\]
In the language of Definition~\ref{Def:RDO}, this is a rough differential operator.
Since $\extd_e u \in \ran(\extd_e) \subset \ker(\extd_e) \subset \dom(\extd_e)$ and similarly, $\omega \wedge u \in \dom(\extd_e)$,  we write
\begin{align*} 
\Dir_e^2 u 
&= 
\extd_e^2 u + \extd_e(\omega \wedge u) + \omega \wedge( \extd_e u) + \omega \wedge \omega \wedge u\\ 
&= 
(\extd_e \omega) \wedge u + (-1)^{k} \omega \wedge (\extd_e u) + \omega \wedge( \extd_e u)+ \omega\wedge\omega\wedge u.
\end{align*}
We see that if $\omega \in \ker(\extd_e)$ and $k$ is odd, then $\Dir_e$ is nilpotent.
Since odd forms are a vector subspace, we assume from here on that 
\[ 
\omega \in \ker(\extd_e) \cap \Lp{0}(\cM; \Forms[\odd]\cM).
\]
The adjoint is then readily verified to be 
\[
\Dir_{e}^{\ast, (\mg,\mu)} v= \extd_e^{\ast,(\mg,\mu)} v+ \omega \cut_{\mg} v, 
\]
for $v \in \dom(\Dir_{e}^{\ast, (\mg,\mu)}) = \dom(\extd_e^{\ast,(\mg,\mu)})$. 
This is again nilpotent by the abstract theory, namely Lemma~\ref{Lem:AbsAdj}.
With this, we can form the operator which we call the \emph{generalised Hodge-magnet}
\[ 
\DDir_{e, \mg,\mu} := \Dir_e + \Dir_{e}^{\ast, (\mg,\mu)},
\]
with the induced intersection domain. 
Letting $\omega \cliff_{\mg} u  = \omega \wedge u + \omega \cut_{\mg}  u$, we see that 
\[ 
\DDir_{e, \mg,\mu} u= (\extd_{e} + \extd_{e}^{\ast, (\mg,\mu)})u + \omega \cliff_{\mg} u = \DDir_{\mathrm{\Hodge},\mg,\mu} u + \omega \cliff_{\mg} u.
\]

Note that our terminology arises from the fact that when $\omega = \imath \eta \in \ker(\extd_{e}) \cap \Lp{0}(\cM;\Forms[2l+1]\cM)$, the operator $\DDir_{\e,\mh,\mu_{\mh}}$ for a smooth  metric $\mh$ is called the magnetic Hodge-Dirac operator and $\DDir_{\e,\mh,\mu_{\mh}}^2$ the magnetic Hodge-Laplacian.

Letting $(\mh,\nu)$ be another RHM-LECMS pair mutually bounded with respect to $(\mg,\mu)$, invoking Theorem~\ref{Thm:GenHodgeMain} yields that 
\[ 
\Hom_{\dR}(\Dir_{e}) = \faktor{ \ker( \Dir_{e})}{ \close{\ran(\Dir_{e})}}  \cong \ker(\DDir_{e,\mg,\mu}) \cong \ker(\DDir_{e,\mh,\nu}).
\]

\subsection{Bundle-valued forms and flat connections}
\label{Sec:BVF}

Let $\cE \to \cM$ be a vector space over a manifold $\cM$ equipped with a connection $\nabla$ and as before, let  $\Ck{\infty}(\cM;\cE)$ the space of smooth sections. 
We define the space of differential forms valued in $\cE$ as $\Forms[l](\cM;\cE):=\Ck{\infty}(\cM;\Forms[l] \cotanb\cM\otimes \cE)$. 
The space of $\cE$ valued forms of all degrees is denoted by $\Forms(\cM;\cE) = \sum_{i=0}^n \Forms[i](\cM;\cE)$.
We define the wedge product $\wedge:\Forms[k](\cM)\times \Forms[l](\cM;\cE)\to \Forms[k+l](\cM,\cE)$ as the map $\alpha\wedge\beta:=(\alpha\wedge\gamma)\otimes\sigma$  whenever $\alpha\in \Forms[k](\cM),\,\, \beta=\gamma\otimes\sigma\in \Forms[l](\cM;\cE)$.
The connection $\nabla:\Forms[0](\cM,\cE)\to \Forms[1](\cM,\cE)$ can be uniquely extended to the linear map $\extd^\nabla: \Forms[l](\cM;\cE)\to \Forms[l+1](\cM;\cE)$ such that $\extd^\nabla$ satisfies the Leibniz rule: for all $f\in \Ck{\infty}(\cM), \alpha\in \Omega^l(\cM;\cE)$
\[
\extd^\nabla(f\alpha)=\extd f\wedge\alpha + f \extd^\nabla\alpha,\,\, \extd^\nabla(\beta\otimes\sigma)=\extd^\nabla\beta\otimes\sigma+(-1)^k \beta\wedge\nabla\sigma, 
\]
with $\beta\in \Forms^k(\cM),\sigma\in \Ck{\infty}(\cM;\cE)$. 
Letting $\Rcur^\nabla$ denote the curvature of the connection $\nabla$, we have 
\[
(\extd^\nabla)^2\alpha=\Rcur^\nabla\wedge\alpha
\] 
for any $\alpha\in \Forms[l](\cM;\cE)$. 
Here the wedge is defined in the following sense: 
\[
\wedge:\Forms[2](\cM,\End(\cM;\cE))\times \Forms[l](\cM;\cE)\to \Forms[l+2](\cM;\cE); (\alpha\otimes A)\wedge(\beta\otimes b)=(\alpha\wedge\beta)\otimes Ab.
\] 
If the curvature of the connection is zero, then $(\extd^\nabla)^2=0$. 
If we equip $\cM$  and $\cE$ with rough metrics $\mg$ and $\mh^{\cE}$ respectively, then this extends to a canonical rough metric $(\mg,\mh^{\cE})$ on $\Forms(\cM;\cE)$.
Fixing another such metric $(\mgt, \mht^{\cE})$ satisfying $\mg \sim \mgt$ and $\mh^{\cE} \sim \mht^{\cE}$, we have
\[
\ker(\extd^\nabla_e + (\extd^\nabla_e)^{\ast,\mg,\mh^{\cE}}) \cong \ker(\extd^\nabla_e +(\extd^\nabla_e)^{\ast,\mgt, \mht^{\cE}})
\]
for a nilpotent extension $\extd^\nabla_e$ of $\extd^\nabla$.

The operators $\extd^\nabla_0$ and $\extd^\nabla_2$, defined similarly as for $\extd_0$ and $\extd_2$ as in Subsection~\ref{Ex:Noncompact}, directly generalise the results there to this setting.
If $\cM$ is compact without boundary, then there is only a unique extension which we again denote by $\extd^\nabla$, and then we obtain that
\[
\Hom_{\dR}(\cE,\nabla) \cong \ker(\extd^\nabla + (\extd^\nabla)^{\ast, \mg, \mh})\quad  \text{and} \quad \Hom[k]_{\dR}(\cE,\nabla) \cong \ker((\extd^\nabla + (\extd^\nabla)^{\ast, \mg, \mh})\rest{\Forms[k]\cE})
\]
for a pair of rough metrics $\mg$ and $\mh$ on $\cM$ and $\cE$ respectively.
Similarly, considering $\Lap_{\mg,\mh} := (\extd^\nabla + \extd^{\nabla})^{\ast, \mg,\mh})^2 = \extd^\nabla(\extd^{\nabla})^{\ast, \mg,\mh} + (\extd^{\nabla})^{\ast, \mg,\mh}\extd^\nabla$, the corresponding generalisations of Subsection~\ref{Ex:HodgeLaplace} hold.
Since the results in Subsection~\ref{Ex:HodgeBoundary} can be considered a special case of the results in Subsection~\ref{Ex:Noncompact}, the setup here directly generalises the results there also.
Furthermore, Hodge-magnets as in Subsection~\ref{Sec:Hodge-magnet} can also be considered at this level of generality. 

\subsection{Bounded elliptic weights}
\label{Eq:EllWeights}
On $\cE \to\cM$ and $\cM$ a smooth manifold equipped with a RHM-LECM pair $(\mh,\mu)$, we can consider \emph{bounded elliptic weights} $f, \frac{1}{f} \in \Lp{\infty}(\cM;\R)$.
For two such weights $f_1$ and $f_2$, we can consider the induced RHM-LECM pair $(f_1\mh, f_2 \mu)$.
By the ellipticity condition, it is clear that $(f_1 \mh, f_2 \mu) \sim (\mh,\mu)$ and therefore, these weights are simply a subclass of the mutually bounded RHM-LECM pairs we consider.

\subsection{Nontriviality of the intersection of the minimal domain and the domain of its adjoint}
\label{Ex:HodgeMinDomain}
For a densely-defined, closed extension $\Dir_{e}$ of a differential operator $\Dir$, it is non-trivial why we can expect the Hodge-Dirac-type operator $\Dir_{e} + \Dir_{e}^\ast$ to also be densely-defined and closed.
While we prove this in a very general setup in Proposition~\ref{Prop:DenselyDefined}, we highlight as to why we expect this to be true through an explicit example.

Let $\cM$ be a smooth compact manifold with smooth boundary $\partial \cM \neq \emptyset$.
Fix a smooth metric $\mg$ and let $\extd: \Ck{\infty}(\cM;\Forms\cM) \to \Ck{\infty}(\cM;\Forms\cM)$ be the exterior derivative. 
Denote the the formal adjoint of $\extd$ with respect to $\mg$ by  $\extd^\dagger$.

The maximal extension $\extd_{\max} = (\extd^\dagger\rest{\Ck[cc]{\infty}(\cM;\Forms\cM)})^\ast$ and the minimal extension is obtained by $\extd_{\min} = \close{\extd\rest{\Ck[cc]{\infty}(\cM;\Forms\cM)}}$.
Note that $\extd_{\min} = \extd_0$ as outlined in Example~\ref{Ex:HodgeBoundary}.
Therefore, $\extd_{\min}$ is, in fact, closed, densely-defined, and nilpotent operator.

Let $\Dir_{0} := \extd_{\min} + \extd_{\min}^\ast$, with domain $\dom(\extd_{\min}) \cap \dom(\extd_{\min}^\ast)$.
From Theorem~\ref{Thm:GenHodgeDiracMain}, we conclude that $\Dir_{0}$ is a self-adjoint operator.
This may seem surprising and on face value, this intersection might be very small.

Let us now consider the formally self-adjoint Hodge-Dirac operator, $\Dir = \extd + \extd^\dagger: \Ck{\infty}(\cM;\Forms\cM)\to \Ck{\infty}(\cM;\Forms\cM)$. 
This is an elliptic first-order differential operator which is of Dirac-type. 
Let the minimal extension be given by  $\Dir_{\min} = \close{\Dir\rest{\Ck[cc]{\infty}(\cM;\Forms\cM)}}$.
The closure of this operator exists since  $\Dir_{\max} = (\Dir\rest{\Ck[cc]{\infty}(\cM;\Forms\cM)})^\ast$ exists by the fact that $\Dir$ is formally self-adjoint. 
Since $\Dir_{0}$ is an extension, we have that $\Dir_{\min} \subset \Dir_{0} \subset \Dir_{\max}$. 

Applying Corollary~6.6 in \cite{BB12}, we obtain that $\dom(\Dir_{\min}) = \SobH[0]{1}(\cM; \Forms\cM)$.
But applying Corollary~3.13 in \cite{BGS}, we have that $\spec(\Dir_{\min}) = \Co$, the entire complex plane.
Since $\Dir_{0}$ is self-adjoint, we must have that $\Dir_{\min} \neq \Dir_0$.
Similarly, Corollary~3.13 in \cite{BGS} shows us that $\Dir_{0} \neq \Dir_{\max}$.
That is $\Dir_{\min} \subsetneqq \Dir_{0} \subsetneqq \Dir_{\max}$.

While this abstract reasoning demonstrates this strict containment, let us construct a particular instance that shows $\Dir_{\min} \neq \Dir_{0}$.
For that, set $\cM := B = \set{ u \in \R^3: \modulus{u} \leq 1}$, the unit ball in $\R^3$.
We let $\mg = \inprod{\cdot,\cdot}_{\R^3}$, the usual Euclidean metric.

To show $\Dir_{\min} \neq \Dir_{0}$, since $\dom(\Dir_{\min}) = \SobH[0]{1}(B;\Forms\R^3)$, we show that there exists $f \in \Ck{\infty}(B;\Forms\R^3) \cap \dom(\Dir_{0})$ with $f\rest{\partial B} \neq 0$.

Let $f \in \Ck{\infty}(B; \Forms\R^3)$ and write in polar coordinates
\begin{equation} 
\label{Eq:f}
f(r,\theta, \phi) = f_{r}(r,\theta,\phi)\ dr + f_{\theta}(r,\theta,\phi)\ d\theta  + f_{\phi}(r,\theta,\phi)\ d\phi. 
\end{equation}
A calculation then yields 
\begin{equation} 
\label{Eq:df}
\begin{aligned}
(\extd f) 
= 
(\partial_r f_{\theta} - \partial_\theta f_r)\ dr \wedge d\theta +
(\partial_r f_{\phi} - \partial_{\phi} f_r)\ dr \wedge d\phi +
(\partial_\theta f_{\phi} - \partial_{\phi} f_\theta)\ d\theta \wedge d\phi.
\end{aligned} 
\end{equation}

We ask for the first component to be zero, that is $\partial_r f_\theta = \partial_\theta f_r$.
This can be achieved if $f_\theta$ and $f_r$ are constant functions in $r$ and $\theta$.
We also would like to remove the difference terms to make calculations easier.
Therefore, for functions $a, b \in \Ck{\infty}(\Sph^1)$  to be chosen later, we make the ansatz
\begin{equation*} 
f_r(r, \theta, \phi) = a(\phi), \quad 
f_\theta(r, \theta, \phi) = b(\phi), 
\quad\text{and}\quad
f_\phi(r,\theta,\phi) \equiv 0. 
\end{equation*} 
Therefore, Equation~\eqref{Eq:df} then simplifies to 
\begin{equation}
\label{Eq:df2}
(\extd f)(r,\theta,\phi) = -a'(\phi)\ dr \wedge d\phi  - b'(\phi)\ d\theta \wedge d\phi.
\end{equation}
It is readily verified that 
\[
\norm{f}_{\extd} \simeq \norm{f}_{\Lp{2}(B)} + \norm{\extd f}_{\Lp{2}(B)} \simeq  \norm{a}_{\SobH{1}(\Sph^1)} + \norm{b}_{\SobH{1}(\Sph^1)}.
\]

In order to show that $f \in \dom(\extd_{\min})$, we need to show that there exists $f_{\epsilon} \in \Ck[cc]{\infty}(B; \Forms\R^3)$ such that $f_\epsilon \to f$ in $\dom(\extd_{\max})$.
For $\epsilon \in (0,1)$, using $\chi_{B(1- \epsilon)}$, the characteristic function on $B(1 - \epsilon)$ or the ball of radius $1 - \epsilon$, through mollification, we can find $\eta_{\epsilon} \in \Ck[cc]{\infty}(B;[0,1])$ such that $\eta_{\epsilon} = 1$ on $B(1 - \epsilon)$ and $\eta_{\epsilon} = 0$ on $B \setminus B(1 - 10^{-1000} \epsilon)$.
Let 
\[
f_{\epsilon}(r,\phi,\theta) := \eta_{\epsilon}(r) a(\phi)\ dr   + \eta_{\epsilon}(r) b(\phi)\ d\theta. 
\]
Then,  via \eqref{Eq:df}, 
\[
(\extd f_{\epsilon})(r,\phi,\theta) = \eta_{\epsilon}'(r) b(\phi)\ dr\wedge d\theta 
- \eta_{\epsilon}(r) a'(\phi)\ dr \wedge d\phi 
- \eta_{\epsilon}(r) b'(\phi)\ dr \wedge d\phi.
\]
The term $\eta_{\epsilon}'$ cannot be controlled uniformly in $r$ as $\epsilon \to 0$, so let us further choose $b \equiv 0$.
In that case 
\begin{align*} 
&f(r,\phi,\theta) =  a(\phi)\ dr  \\ 
&f_{\epsilon}(r,\phi,\theta) = \eta_{\epsilon}(r) a(\phi)\ dr  \\ 
&(\extd f)(r,\phi,\theta) = (\extd f)(r,\theta,\phi) = -a'(\phi)\ dr \wedge d\phi\\
&(\extd f_{\epsilon})(r,\phi,\theta) =- \eta_{\epsilon}(r) a'(\phi)\ dr \wedge d\phi.
\end{align*}
Therefore,
\[ 
\norm{f - f_{\epsilon}}_{\extd_{\max}}
\simeq 
\norm{( 1 - \eta_{\epsilon})a }_{\Lp{2}(\Sph^1)} + \norm{( 1 - \eta_{\epsilon})a' }_{\Lp{2}(\Sph^1)}
\] 
and clearly, via the Dominated Convergence Theorem, the right hand side tends to $0$ as $\epsilon \to 0$.
This shows that $f \in \dom(\extd_{\min})$ but $f(1,\theta,\phi) = a(\phi)\ dr \neq 0$ if $a \in \Ck{\infty}(\Sph^1)$ is chosen such that $a \neq 0$.
In polar coordinates, $f \rest{\partial B} = f(1,\cdot,\cdot) \neq 0$. Since $\dom(\Dir_{0}) = \dom(\extd_{\min}) \cap \dom((\extd_{\min})^\ast)$, we need to still show that $f \in \dom((\extd_{\min})^\ast)$. 
For that, first note $(\extd_{\min})^\ast = \extd^\dagger_{\max}$. 
But $\Ck{\infty}(\cM; \Forms\R^3) \subset \dom(\extd^\dagger_{\max})$ and this shows that $f \in \dom((\extd_{\min})^\ast)$.
In fact, 
\[
\extd^\dagger f = da \cut\ dr = (\partial_r a\ dr + \partial_\theta a\ d\theta + \partial_\phi a\ d\phi) \cut\ dr = 0.
\] 
So we have that $f \in \ker((\extd_{\min})^\ast)$. Together, we have shown $f\rest{\partial B} \neq 0$ but with $f \in \dom(\Dir_{0})$. Note that since we have freedom to choose $a \in \SobH{1}(\Sph^1)$, we can equally as well take $a \equiv \lambda \in \C$, a constant.
In that case, we see that $f \in \ker(\extd_0) \cap \ker(\extd_0^{\ast,\mg}) = \ker(\Dir_0)$.
Therefore, even  solutions to $\Dir_0 u = 0$ might have a non-zero restriction to the boundary.
Compare this with $\Dir_{\min}$, where by construction, all solutions $\Dir_{\min} u = f$ must necessarily vanish on the boundary.  We further remark that in Sections~7~and~10 in \cite{RosenBook}, for bounded Lipschitz domains in Euclidean space $\Omega \subset \R^n$, the domain of $\extd_{\min}$ is characterised precisely as the forms which have a vanishing tangential term at the boundary.
The calculation here produces an explicit example of such a form.

\subsection{Non-nilpotency of a closed extension}
\label{Sec:Nonnil} 

In Definition~\ref{Def:Nilpotent}, the notion of a nilpotent extension is characterised.
We show here that this is a non-trivial definition.
That is, even in the case when $\mg$ is smooth and $\extd^2 =0$ on smooth forms, there are extensions of $\extd$ which are non-nilpotent.

Firstly, let us carry forth the setup for the previous subsection, Subsection~\ref{Ex:HodgeMinDomain}. 
That is, we fix $\cM := B \subset \R^3$, the unit ball in $\R^3$. In polar coordinates, we have seen that $f \in \Ck{\infty}(B;\Forms[1]\R^3)$ is described as in equation~\eqref{Eq:f} and $\extd f$ as in equation~\eqref{Eq:df}. Let us define the operator $\extd_{\theta}$ as the closure of $\extd_{\theta,c}$, where
\[  
\dom(\extd_{\theta,c}) := \set{ u\in \Ck{\infty}(B;\Forms \R^3): d\theta \cut (u \rest{\partial B}) = 0}.
\]
Note that $d\theta \cut f = f_{\theta}$  since $f \in \Ck{\infty}(B;\Forms[1]\R^3)$ and therefore, the condition $d\theta \cut (f \rest{\partial B}) = 0$ occurs if and only if $f_{\theta}(1,\theta,\phi) = 0$.
To fix a particular $f$, let us set $f_{r}, f_{\theta} \equiv 0$ and $f_{\phi} (r,\phi,\theta) = a(\theta)$, where $a \in \Ck{\infty}(\Sph^1)$.
Therefore, 
\[ 
f(r,\theta,\phi) = a(\theta)\ d\phi \quad \text{and}\quad (\extd f)(r,\theta,\phi) = a'(\theta)\ d\theta \wedge d\phi.
\]
Since 
\[ 
d\theta \cut (\extd f)\rest{\partial B} = - a'(\theta)\ d\phi
\]
we ensure that we further  choose $a'(\theta) \neq 0$, i.e. $a(\theta)$ a constant.  This guarantees that $\extd_{\theta,c}$ is not a nilpotent extension, but its closure $\extd_{\theta}$ could very well be.
However, in what is to follow, we show that $\extd_{\theta}$ also fails to be nilpotent. To show that $\extd_{\theta}$ is not a nilpotent extension, by Corollary~\ref{Cor:NilExt0}, it suffices to show that $\extd f \not \in \dom(\extd_{\theta})$. 
We show this through contradiction. To that end, suppose that $v_n \to \extd f$ in the graph norm of $\extd$ with $v_n \in \dom(\extd_{\theta,c})$. Since $\extd f \in \Ck{\infty}(B;\Forms[2]\R^3)$, we can further assume that $v_n \in \Ck{\infty}(B; \Forms[2]\R^3)$.
Let us write
\begin{equation*} 
\label{Eq:v2}
v_n(r,\theta,\phi) 
:= v_{n,r\theta}\  dr \wedge d\theta + v_{n,r\phi}\ dr \wedge d\phi + v_{n,\theta \phi}\ d\theta \wedge d\phi, 
\end{equation*}
and so $v_n \in \dom(\extd_{\theta,c})$ means precisely that 
\[
0 = d\theta \cut  (v_n)\rest{\partial B} = - v_{n,r\theta}\rest{\partial B}\  dr +  v_{n,\theta \phi}\rest{\partial B} d\phi. 
\]
That is, $v_{n,r\theta}\rest{\partial B} = 0$ and $v_{n,\theta \phi}\rest{\partial B} = 0$ and therefore,
\[
v_n\rest{\partial B} = v_{n,r\phi}\rest{\partial B}\ dr \wedge d\phi. 
\]
In particular, this implies that $dr \wedge (v_n\rest{\partial B}) = 0$, which means that $v_n \in \dom(\extd_{\min})$ by Sections 7 and 10 in \cite{RosenBook}, where we have defined $\extd_{\min}$ in the previous Subsection~\ref{Ex:HodgeMinDomain}. 
That means that $\extd f \in \dom(\extd_{\min})$.
By the same results of \cite{RosenBook}, we would then have that
\[ 
0 = dr \wedge \extd f\rest{\partial B} = dr \wedge (a'(\theta)\ d\theta \wedge d\phi) = a'(\theta)\ dr \wedge d\theta \wedge d\phi
\]
which implies that $a'(\theta) = 0$. 
However, we chose $a$ such that  $a'(\theta) \neq 0$ which yields the contradiction.
Therefore, $\extd f \not \in \dom(\extd_{\theta})$ and so by Corollary~\ref{Cor:NilExt0}, the extension  $\extd_{\theta}$ is not nilpotent, despite $\extd$ being a smooth coefficient operator satisfying $\extd^2 = 0$ on $\Ck{\infty}(B;\Forms \R^3)$.

\subsection{Non-smooth metrics and rough differential operators}
\label{Sec:NonsmoothDom}

The notion of a rough differential operator is outlined in Definition~\ref{Def:RDO}.
This definition is motivated by the exterior derivative, which for a smooth manifold always acts on smooth sections as it is independent of the metric. 
However, if it is perturbed additively while maintaining nilpotency (see Subsection~\ref{Sec:Hodge-magnet}), it can map smooth sections to measurable sections. 
This makes it a ``rough'' operator, but it is still capable of acting on smooth sections. Since nilpotency is preserved for the adjoint operator, it might be expected that the rough differential operator property is also preserved for the adjoint.
However, we show here that this is not the case. 

To that end, fix $\cM := \R^2$ and let a RRM on $\R^2$ be given by
\[ 
\mg (x_1,x_2) = \begin{pmatrix}
		(2 + w(x_1))^{\frac 23} & 0 \\
		0 & 1
		\end{pmatrix}, 
\]
where $w: \R \to [-1,1]$ is the Weierstrass function defined by 
\[ 
w(x) := \sum_{k=0}^\infty 2^{-k} \cos ( 14^k \pi x).
\]
Recall that this is an everywhere continuous but nowhere differentiable function. 
As deduced from the proof of Theorem~\ref{Thm:GenHodgeDiracMain}, we obtain $\extd^{\ast,\mg} = B^{-1} \extd^{\ast} B$, and the coefficient matrix on $\Forms[1]\cM$ is given by 
\[
B(x_1,x_2)\rest{\Forms[1]\R^2}
= \mg(x_1,x_2) \sqrt{\det \mg(x_1,x_2)}
= \begin{pmatrix}
		2 + w(x_1) & 0 \\
		0 & (2 + w(x_1))^{\frac 13}
\end{pmatrix}.
\]
Suppose now that $\Ck[c]{\infty}(\cM;\Forms \R^2) \subset \dom(\extd^{\ast,\mg})$.
In particular, we would get that $ \omega(x) := \eta(x)\ dx^1 \in \dom(\extd^{\ast,\mg})$, where $\eta \in \Ck[c]{\infty}(\R^2)$ with $\eta = 1$ on the Euclidean unit ball $B_{1}(0)$.
However, from Proposition~\ref{Prop:HilAdjChange}, we have that $\dom(\extd^{\ast,\mg})  = B^{-1} \dom(\extd^{\ast})$ or equivalently, $B\, \dom(\extd^{\ast,\mg})  = \dom(\extd^{\ast})$. 
In particular, we have that  $B(x_1,x_2) \omega = (2 + w(x_1)) \eta(x_1, x_2)\ dx^1 \in \dom(\extd^{\ast})$.
For every function $f \in \Ck[c]{\infty}(B_1(0))$, we have that 
 \[
\inprod{\extd^{\ast}B\omega, f} = \inprod{B\omega, \extd f} = \int_{\R^2} (2 + w(x_1)) \eta(x_1,x_2) \partial_1 f\ dx = \int_{B_1(0)} (2 + w(x_1)) \partial_1 f\ dx
\]  
since $\eta = 1$ on the unit ball.
By further considering functions $f = (x_1, x_2) \mapsto h(x_1)\xi(x_2)$ where $h \in \Ck[c]{\infty}(\R)$ with $\spt h \subset [-1,1]$ and $\xi \in \Ck[c]{\infty}(\R)$ with $\xi = 1$ on $[-1,1]$, we deduce that  $y \mapsto 2 + w(y)$ is weakly differentiable  on $[-1,1]$ and hence, $y \mapsto w(y)$ is weakly differentiable on $[-1,1]$.
 However, it is known that the Weierstrass function is nowhere differentiable and, therefore, this leads to a contradiction. Thus, $\Ck[c]{\infty}(\cM;\Forms \R^2)$ is not a subset of  $\dom(\extd^{\ast,\mg})$.  By a similar argument, it is easy to see that $\Ck[c]{\infty}(\cM; \Forms \cM)$ is not in general a subspace of $\dom(\Dir_{\mg})$ when the metric $\mg$ is non-smooth.

\section{Generalised Hodge theorem for rough differential operators}
\label{Sec:GH}

\subsection{Measurability and rough differential operators}

Throughout, let $\cM$ be a manifold with a smooth differential structure.
We  allow  $\cM$ to either have boundary or be boundaryless. 

Recall the sets $\Meas$ and $\Zero$ from Definition~\ref{Def:MeasZero}.
These sets, as the names suggest, are $\sigma$-algebras as demonstrated below.
They will be used as an intrinsic instrument to discuss measurability and also to establish the notion of a property holding almost-everywhere. 
To begin with, we note the following. 

\begin{proposition}
\label{Prop:MeasZero}
For any smooth Riemannian metric $\mg$ on $\cM$, we have that $\Meas = \Meas(\mu_\mg)$ and $\Zero = \Zero(\mu_\mg)$ where $\mu_{\mg}$ is the induced Riemannian measure.
In particular, the collections $\Zero$ and $\Meas$ are $\sigma$-algebras satisfying $\Zero \subset \Meas$.
\end{proposition}
\begin{proof}
If $A \in \Meas$, on fixing $(\psi, U)$ with $U$ precompact, we have that $\mu_{\mg} = \sqrt{\det \mg}\ \psi^{\ast}\Leb$ inside $U$.
By the precompactness of $U$, $\sqrt{\det \mg}$ is bounded on $U$ by the smoothness of $\mg$ and so it is clear that $A \cap U$ is $\mu_\mg$-measurable.
The set $A$ can then be covered by countable such $A \cap U$ and therefore, $A$ is $\mu_\mg$-measurable. 
Conversely, for $B \in \Meas(\mu_{\mg})$, by covering an arbitrary $(\psi,U)$ by countably many precompact $V$, we can see that $\psi(A \cap U)$ is Lebesgue measurable. The assertion $\Zero = \Zero(\mu_\mg)$ follows by covering by countably many compact charts $\set{U_i}$ and using the fact that there exists $C_i \geq 1$ for each $U_i$ such that $C_i^{-1} \leq  \det \mg(x) \leq C_i$ for all $x \in U_i$. In particular, this yields that $\Meas$ and $\Zero$ are $\sigma$-algebras with $\Zero \subset \Meas$. 
\end{proof}

Instead of defining $\Meas$ and $\Zero$ as we did in Definition~\ref{Def:MeasZero}, we could have first proved this proposition and simply defined it to be the measurable sets and zero measure sets with respect to a background measure induced by a metric. 
However, our approach has the added benefit of emphasising that these notions are independent of any metric and they are actually imported from the differentiable structure of the manifold itself.

It is via  Proposition ~\ref{Prop:MeasZero} that we are able to formulate the notion of a measurable section as in Definition~\ref{Def:MeasSec} over a bundle $\cE \to \cM$.
Recall that  a section $\xi: \cM \to \cE$ is measurable if the coefficients of $\xi$ inside any local trivialisation is measurable. 
Moreover, we denote the measurable sections of $\cE \to \cM$ by $\Lp{0}(\cM;\cE)$.
With this, we can formulate the notion of a locally elliptic compatible measure (LECM) as well as a rough Hermitian metric (RHM) (see Definition~\ref{Def:LECM} and \ref{Def:RHM}).
A useful observation for such metrics is the following. 

\begin{lemma}
\label{Lem:LocComp}
Let $\mh$ be a RHM and $\mh'$ any smooth metric. 
Then, on every compact $K \subset \cM$, there exists a constant $C_K \geq 1$ such that
$$\frac{1}{C_K} \modulus{u}_{\mh'(y)} \leq \modulus{u}_{\mh(y)} \leq C_K \modulus{u}_{\mh'(y)}$$
for $y$-almost-everywhere in $K$ and for all $u \in \cE_y$.
\end{lemma}
\begin{proof}
Fix $\mh'$ a smooth Hermitian metric and $K$ a compact subset of $\cM$.
Cover $K$ by finitely many local compatible structures $(U_1, \mh_{1}^\infty), \cdots , (U_M, \mh_{M}^\infty)$. 
Clearly, since $\mh'$ and $\mh_{i}^\infty$ are smooth, $\mh' \sim \mh_{i}^\infty$ on $K \cap U_i$.
Since by definition, $\mh \sim \mh_{i}^\infty$ on $K \cap U_i$ and $\sim$ is transitive, it is clear that $\mh \sim \mh'$ on $K \cap U_i$.
By compactness of $K$, the conclusion follows. 
\end{proof} 

Next we consider function spaces induced by the measure and geometry.
For this,  let us fix a LECM measure $\mu$ and RHM $\mh$ and recall that we call $(\mu,\mh)$ a LECM-RHM pair.
Note that the measures we consider are fixed so that their measure theory coincides with the intrinsic notions of measure and zero content that we have defined.

To consider the results we present in this paper, we need to ensure that rough differential operators are densely-defined as operators acting on $\Lp{2}(\cM;\cE,\mu,\mh)$ whenever $(\mu,\mh)$ is a LECM-RHM pair.
For a smooth pair $(\mu',\mh')$, this follows readily since a rough differential operator acts smooth functions and in particular, it acts on $\Ck[cc]{\infty}(\cM;\cE)$ which is dense in $\Lp{p}(\cM;\cE,\mu',\mh')$. 
Therefore, it suffices to ensure that this latter density property survives into the more general setting. 
To that end, we begin with the following Lemma.

\begin{lemma}
\label{Lem:RDBound}
For any two LECMs $\mu_1$ and $\mu_2$, we have that 
$$ \mu_1 = \frac{d\mu_1}{d\mu_2} \mu_2.$$
On each compact $K \subset \cM$, $\mu_1$ and $\mu_2$ are mutually bounded on $K$. 
\end{lemma}
\begin{proof} 
Note $\mu_1$ and $\mu_2$ are absolutely continuous with respect to each other since $\Zero(\mu_1) = \Zero = \Zero(\mu_2)$. 
So the expression for $\mu_1$ in terms of $\mu_2$ simply follows from the Radon-Nikodym Theorem. Fix $K$ and cover it by finitely many $(U,\psi)$ charts as in Definition~\ref{Def:LECM}, say $(U_1,\psi_1), \cdots, (U_l, \psi_l)$.
We can choose these charts so that the condition in Definition~\ref{Def:LECM} is satisfied for both $\mu_1$ and $\mu_2$.
Clearly, we can then choose a constant $C_K$ so that 
$$ \frac{1}{C_K} \leq \frac{d\mu_i}{d\psi_i^\ast (\Leb \restrict \psi_i(U_i)) } \leq C_K.$$
Moreover, $\mu_j \restrict U_i \sim \psi_i^{\ast}(\Leb \restrict \psi_i(U_i)$, we can write 
 $$\frac{d\mu_1}{d\mu_2}\rest{U_i} 
= \frac{d\mu_1}{d\psi_i^\ast (\Leb \restrict \psi_i (U_i)) }  \frac{d\psi_i^\ast (\Leb \restrict \psi_i (U_i)) } {d\mu_2}  
= \frac{d\mu_1}{d\psi_i^\ast (\Leb \restrict \psi_i (U_i)) }  \cbrac{\frac {d\mu_2}{d\psi_i^\ast (\Leb \restrict \psi_i (U_i)) }}^{-1},
$$
which is clearly bounded from above. 
This, together with a similar estimate for $\frac{d\mu_2}{d\mu_1}\rest{U_i}$ along with the compactness of $K$, yields the conclusions. 
\end{proof} 

With this, we prove the following important density result. 

\begin{proposition} 
For a LECM-RHM pair $(\mu,\mh)$, the subspace $\Ck[cc]{\infty}(\cM;\cE)$ is dense in $\Lp{p}(\cM;\cE,\mu,\mh)$ for $1 < p < \infty$. 
\end{proposition}
\begin{proof}
Fix $u \in \Lp{p}(\cM;\cE,\mu,\mh)$ and  write $\cM = \union_{j=1}^\infty U_j$ where $U_j \subset U_{j+1}$ and each $U_j$ is precompact.  
Let $\chi_j$ be the indicator function on $U_j$.
Then, 
 $$ f_j := \modulus{\chi_j u}^p_{\mh} \leq \modulus{u}_{\mh}^p$$
and,  hence, $(f_j)_j$  is an increasing sequence of measurable functions, compactly supported,  with $\modulus{f_j} \leq f := \modulus{u}_{\mh}^p$. The latter which has a finite integrand with respect to $\mu$.  
Clearly, $f_j \to f$ pointwise almost-everywhere.
Therefore, by the dominated convergence theorem, we have that  $ \chi_j u \to u$ in $\Lp{p}(\cM;\cE,\mu,\mh)$.  Given an $\epsilon > 0$,  we fix   $j$ large enough so that $\norm{\chi_j u - u}_{\Lp{p}} \leq \nicefrac{\epsilon}{10}$.
 Since  $\close{U_j}$ is compact and $h$ is a rough metric, then we can find smooth metrics $\mh_j^\infty$ such that $\mh_j^\infty\rest{U_j}  \sim \mh\rest{U_j}$ inside $U_j$. From Lemma~\ref{Lem:RDBound},  we have that $\mu_{\mh_j^\infty} \restrict U_j \sim \mu \restrict U_j$ where  $\mu_{\mh_j^\infty}$ is the smooth measure induced from the smooth metric $\mh_j^\infty$. 
Therefore, $\chi_j u \in \Lp{p}(U_j;\cE,\mu_{\mh_j^\infty}, \mh_j^\infty)$.
Then, there exists $(v_j)_m \to \chi_j u$ in $\Lp{p}(U_j;\cE,\mu_{\mh_j^\infty}, \mh_j^\infty)$ with $(v_j)_m \in \Ck[c]{\infty}(U_j;\cE)$. Moreover, from these bounds, we have a constant $C_j$ such that 
$$ \frac{1}{C_j} \norm{\cdot}_{\Lp{p}(U_j;\cE,\mu, \mh)} \leq \norm{\cdot}_{\Lp{p}(U_j;\cE,\mu_{\mh_j^\infty}, \mh_j^\infty)} \leq C_j \norm{\cdot}_{\Lp{p}(U_j;\cE,\mu, \mh)}.$$
Therefore, let $m$ be sufficiently large so that 
$$\norm{ (v_j)_m  - \chi_j u}_{\Lp{p}(U_j;\cE,\mu_{\mh_j^\infty}, \mh_j^\infty)} \leq \frac{\epsilon}{10 C_j}.$$
Extending $(v_j)_m$ by zero outside of $U_j$, we have that $(v_j)_m \in \Ck[c]{\infty}(\cM;\cE)$ and 
$$ \norm{ (v_j)_m - u}_{\Lp{p}(U_j;\cE,\mu, \mh)} \leq  \norm{ (v_j)_m - \chi_j u}_{\Lp{p}(U_j;\cE,\mu, \mh)} + \norm{u - \chi_j u}_{\Lp{p}(U_j;\cE,\mu, \mh)}
\leq C_j \frac{\epsilon}{10 C_j} + \frac{\epsilon}{10}  < \epsilon.$$
This proves the claim. 
\end{proof}

With this, we have the following important consequence. 
\begin{corollary}
Every rough differential operator $\Dir$ is densely-defined in $\Lp{p}(\cM;\cE, \mu,\mh)$ for $1 \leq p < \infty$ and all LECM-RHM pairs $(\mu,\mh)$. 
\end{corollary}

\subsection{Mutually bounded pairs}

Our ultimate goal is to compare kernels of operators built out of rough differential operators which records measure and metric information.
We proceed by identifying a large class of metrics that yield an equivalent $\Lp{2}$-theory.
This is captured in the  definition of a mutually bounded pair from Definition~\ref{Def:MutBdd}. To proceed, we first prove the following important change of metric formula.
\begin{proposition}
\label{Prop:MutMat}
Let $\mh_1$ and $\mh_2$ be two RHMs (not necessarily mutually bounded) on $\cE \to \cM$.
Then, there exists $A \in \Lp{0}(\cM;\Sym_{\C}\End(\cE))$ such that $\mh_1(x)[u,v] = \mh_2(x)[A(x)u,v]$ for $x$-a.e. and for all $u,v \in \cE_x$.
If $\mh_1 \sim \mh_2$ with constant $C \geq 1$, then for almost-every $x \in \cM$,  
$$ \frac{1}{C^{2}} \leq \modulus{A(x)}_{\mh_i(x)} \leq C^2.$$
\end{proposition}
\begin{proof} 
We begin by first noting that for a finite dimensional complex vector space $V$, if $a$ and $b$ are two inner products, then there exists an invertible,  Hermitian symmetric endomorphism $X_{ab}$ with $a(u,v) = b(X_{ab}u,v)$. 
Fix $R \subset \cM$ to be the intersection of the regular sets for $\mh_1$ and $\mh_2$ (i.e., where $\mh_i(x) \neq \infty$ and $\mh_i(x) \neq 0$ and $\mh_i(x)$ is Hermitian symmetric).
Clearly, $\cM \setminus R$ is of measure zero. 
Now, for $x \in R$, we set $V = \cE_x$ and by the previously mentioned result, we obtain $\mh_1(x)[u,v] = \mh_2(x)[A(x)u,v]$ for all $u,v \in \cE_x$.
Clearly $x \mapsto A(x) \in \Lp{0}(\cM;\Sym_{\C}\End(\cE))$ since $\mh_i \in \Lp{0}(\cM; \Sym_{\C} \cE^\ast \tensor \cE^\ast)$. If $\mh_1 \sim \mh_2$ with constant $C \geq 1$ we have for $u \in \cE_x$ and $v \in \cE_x\setminus\set{0}$
$$ \frac{1}{C} \frac{\modulus{\mh_1(x)[u,v]}}{\modulus{v}_{\mh_1(x)}} \leq \frac{\modulus{\mh_2(x)[A(x)u,v]}}{\modulus{v}_{\mh_2(x)}} \leq C \frac{\modulus{\mh_1(x)[u,v]}}{\modulus{v}_{\mh_1(x)}}$$
since $\mh_1(x)[u,v] = \mh_2(x)[A(x)u,v]$. 
Taking a supremum over $v \neq 0$, we obtain
 $$ \frac{1}{C} \modulus{u}_{\mh_1(x)} \leq \modulus{A(x)u}_{\mh_2(x)} \leq C\modulus{u}_{\mh_1(x)}.$$ 
Again, since $\mh_1 \sim \mh_2$ with $C$, we have the desired claim for $\mh_2$. 
Running the same argument with $\mh_1$ and $\mh_2$ interchanged and instead writing $\mh_2(x)[u,v] = \mh_1(x)[A^{-1}(x)u,v]$, we obtain the desired estimate for $\mh_1$ with $A^{-1}$ in place of $A$. 
However, since it is bounded above and below, it is easy to see that the desired estimate also holds for $A$ with respect to $\mh_1$.
\end{proof}

As the following examples show, these are very natural notions that can arise in a multitude of contexts.
The following illustrates a natural way in which an RHM can arise.

\begin{example}
Let $\mh'$ be a smooth Hermitian metric on $\cE \to \cM$ and  $A \in \Lp{0}(\cM; \End(\cE))$. 
Suppose further that $A \in \Lp[loc]{\infty}(\cM;\End(\cE))$ and that $A^{-1} \in \Lp[loc]{\infty}(\End(\cE))$.
Then, 
$$ \mh_A(x)[u,v] := \mh'(x)[A(x) u, A(x)v]$$ 
defines a RHM.
If $A, A^{-1} \in \Lp{\infty}(\cM; \End(\cE), \mh')$, then $\mh_A$ and $\mh'$ are mutually bounded RHMs.
\end{example} 

In the next example, we see that mutually bounded LECM pairs also arise naturally.

\begin{example}
Let $\mg_1$ and $\mg_2$ be two smooth metrics and suppose there is a constant $C \geq 1$ such that 
 $$ \frac{1}{C} \modulus{u}_{\mg_1(x)} \leq \modulus{u}_{\mg_2(x)} \leq C \modulus{u}_{\mg_1(x)}$$
for all $x \in \cM$ and for all $u \in \tanb_x \cM$.
Then, there exists $A \in \Ck{\infty}(\cM; \Tensors[1,1]\cM)$ such that $\mg_1(u,v) = \mg_2(Au,v)$ and  $\mu_{\mg_1} = \sqrt{ \det A}\ \mu_{\mg_2}.$
The volume measures $\mu_{\mg_1}$ and $\mu_{\mg_2}$ are mutually bounded since 
$$ \frac{1}{C^{\frac n 2}} \leq \sqrt{\det A} \leq C^{\frac n 2}.$$ 
\end{example}

As aforementioned, to compare the kernels of two different operators, we need   at least to know that both operators    can act on a sufficiently large common space. 
In fact, we can do better than that. 
As the following proposition shows, mutually bounded pairs yield the same $\Lp{2}$-theory (up to equivalence).

\begin{proposition}
\label{Prop:LpFix}
Let $(\mu_1,\mh_1)$ and $(\mu_2,\mh_2)$ be two mutually bounded RHM-LECM pairs.
Then: 
\begin{enumerate}[label=(\roman*)] 
\item 
\label{Itm:LpFix:1} 
$\Lp{2}(\cM;\cE) := \Lp{2}(\cM;\cE,\mu_1, \mh_1) = \Lp{2}(\cM;\cE,\mu_1, \mh_1)$ as sets with equivalence of norms ${\norm{\cdot}_{\Lp{2}(\cM;\cE,\mu_1, \mh_1)}\simeq \norm{\cdot}_{\Lp{2}(\cM;\cE,\mu_2, \mh_2)}}.$
\item
\label{Itm:LpFix:2} 
An extension $\Dir_e$ of $\Dir_{\cc}$ closed with respect to $(\mu_1,\mh_1)$ if and only if it is closed with respect to $(\mu_2,\mh_2)$.
\end{enumerate}
\end{proposition}
\begin{proof}
The almost-everywhere pointwise uniform bounded of $\mh_1$ and $\mh_2$, along with a similar assertion for the measures, establishes \ref{Itm:LpFix:1}.
Clearly, \ref{Itm:LpFix:2} then follows.
\end{proof} 

\begin{remark}
It is important to note that, even though $\Dir_{e}$ remains unchanged between the two norms induced by $(\mu_1,\mh_1)$ and $(\mu_2, \mh_2)$, the adjoint operators $\Dir_{e}^{\ast, (\mu_1,\mh_1)}$ and $\Dir_{e}^{\ast, (\mu_2,\mh_2)}$ are in general distinct. 
They encode underlying measure-geometric information associated with $(\mu_i, \mh_i)$.
\end{remark} 


\subsection{Nilpotency and Hodge theory}

Of paramount importance in our considerations is that of a nilpotent operator (c.f. Definition~\ref{Def:Nilpotent}).
Fix $(\mu,\mh)$ a LECMS-RHM pair and $\Dir_e:\dom(\Dir_e) \subset \Lp{2}(\cM;\cE,\mu,\mh) \to \Lp{2}(\cM;\cE,\mu,\mh)$ a closed and nilpotent extension associated to a rough differential operator $\Dir$. 
Since $\Dir_{e}$ is closed, so is $\ker(\Dir_e)$ and therefore, nilpotency self-improves to  $\close{\ran(\Dir_e)} \subset \ker(\Dir_e)$. 
 
In this case, the adjoint $\Dir_e^\ast$ is again nilpotent, but it is not possible to always assert that it is a rough differential operator.
Recall that the domain of the adjoint $\Dir_{e}^\ast$ to $\Dir_e$ is
$$\dom(\Dir_e^\ast) = \set{u \in \Lp{2}(\cM;\cE,\mu,\mh): \modulus{\inprod{\Dir_e v, u}} \lesssim \norm{v}\qquad \forall v \in \dom(\Dir_e)},$$
Therefore, we can see there is no way to assert that $\Ck[c]{\infty}(\cM;\cE) \subset \dom(\Dir_e^\ast)$  for a non-smooth RHM even if $\Dir_e$ itself is a smooth coefficient differential operator.
This is highlighted by an explicit example in Subsection~\ref{Sec:NonsmoothDom}.

Nevertheless, the operator $\DDir_e := \Dir_e + \Dir_e^\ast$ is, densely-defined and self-adjoint.
This is asserted in Proposition~\ref{Prop:DenselyDefined} but  the example in Subsection~\ref{Ex:HodgeMinDomain} concretely illuminates why  we can expect this abstract result to be true.  
More remarkably, this allows an operator-dependent decomposition of the underlying $\Lp{2}$-theory, generalising the famous Hodge-decomposition for the Hodge-Dirac operator. 
We emphasise that these assertions are not our original contributions.
As far as the authors are aware, these results were first established  (in fact, in a more general setup), by Axelsson (Rosén)-Keith-McIntosh in \cite{AKMc}.
Nevertheless, for completeness, we provide a proof of these claims in Appendix~\ref{Appendix}.

\begin{proposition}[Hodge decomposition]
\label{Prop:HodgeDecomp} 
Let $(\mu,\mh)$ be a LECM-RHM pair and $\Dir_e$ be a closed nilpotent extension.
Then $\DDir_{e} := \Dir_e + \Dir_e^\ast$ is self-adjoint and 
$$  \Lp{2}(\cM;\cE,\mu,\mh) = \ker(\Dir_e) \cap \ker(\Dir_e^\ast) \stackrel{\perp}{\oplus} \close{\ran(\Dir_e)} \stackrel{\perp}{\oplus} \close{\ran(\Dir_e^\ast)}.$$
Moreover, $\ker(\DDir_e) = \ker(\Dir_e) \cap \ker(\Dir_e^\ast)$ and $\close{\ran(\DDir_e)} = \close{\ran(\Dir_e)} \stackrel{\perp}{\oplus} \close{\ran(\Dir_e^\ast)}$, where the $\perp$ is with respect to the induced metric $\inprod{\cdot,\cdot}_{(\mu,\mh)}$ from ${(\mu,\mh)}$.
\end{proposition}
\begin{proof} 
Setting $\Hil = \Lp{2}(\cM;\cE,\mu,\mh)$ with the induced metric, on setting $\Gamma := \Dir_e$, we simply invoke Proposition~\ref{Prop:DenselyDefined}.
\end{proof} 

We now provide the proof of Theorem~\ref{Thm:GenHodgeMain}.

\begin{proof}[Proof of Theorem~\ref{Thm:GenHodgeMain}]
Fix two mutually bounded LECM-RHM pairs $(\mu_1,\mh_1)$ and $(\mu_2,\mh_2)$.
By Proposition~\ref{Prop:LpFix}, we have that the $\Lp{2}(\cM;\cE,\mu_1,\mh_1) = \Lp{2}(\cM;\cE,\mu_2,\mh_2) =: \Lp{2}(\cM;\cE)$ with mutually bounded norms. The proof of \ref{Itm:GenHodgeMain:1} is simply a consequence of Proposition~\ref{Prop:HodgeDecomp}.  To  obtain \ref{Itm:GenHodgeMain:2}, we invoke Lemma~\ref{Lem:AlphaDecomp} with a choice of $\Hil = \Lp{2}(\cM;\cE,\mu_i,\mh_i)$. For the part \ref{Itm:GenHodgeMain:3}, we set $\Hil := \Lp{2}(\cM;\cE)$. 
Note that we have $B \in \Lp{0}(\cM; \Sym_{\C} \cE^\ast \tensor \cE^\ast)$ with $B,B^{-1}: \Lp{2}(\cM;\cE) \to \Lp{2}(\cM;\cE)$ such that 
\[ 
\inprod{u,v}_{(\mu_2,\mh_2)} = \inprod{Bu,v}_{(\mu_1,\mh_1)}.
\]
Therefore we have two inner products related through a bounded invertible $B$.
This satisfies the hypotheses of Theorem~\ref{Thm:AbsHodgeType}, and then 
\ref{Itm:GenHodgeMain:3} is shown. Lastly, invoking Corollary~\ref{Cor:AbsHodgeTypeSplittingInt} yields \ref{Itm:GenHodgeMain:4}.
\end{proof}

\begin{remark}
Note that the isomorphism $B$ is constructed from $A$ satisfying $\mh_2(x)[u,v] = \mh_1(x)[A(x)u,v]$ and the Radon-Nikodym derivative $\frac{d\mu_1}{d\mu_2}$ of the measures $\mu_1$ and $\mu_2$. 
The fact that $B$ is an isomorphism arises from the mutual boundedness $\mh_1 \sim \mh_2$ and $\mu_1 \sim \mu_2$.
Moreover, from Proposition~\ref{Prop:HilAdjChange},  we have  $\DDir_{e,2} = \Dir_{e} + B^{-1} \Dir_{e}^{\ast,(\mu_1, \mh_1)} B.$
\end{remark} 

\subsection{Smooth coefficient first-order differential operators}
\label{Sec:SmoothCoeff}

We begin with the notion of a first-order differential operator in the rough context. 
\begin{definition}
\label{Def:FO} 
A rough differential operator $\Dir$ is said to be first-order if for all $\eta \in \Ck{\infty}(\cM)$ commutator $[\Dir, \eta \ident]$ is a multiplication operator.
\end{definition}

A natural way such operators arise is through the addition of a measurable potential. 
For instance, if $\Dir_0$ is smooth coefficient first-order differential operator and $V \in \Lp{0}(\cM;\End(\cE))$ (i.e., a measurable bundle endomorphism), then the operator $\Dir := \Dir_0 + V$ is a first-order rough differential operator. 
Note that in this case, $[\Dir,\eta \ident] = [\Dir_0, \eta \ident] = \sym_{\Dir}(\extd \eta)$, the (smooth) principal symbol of $\Dir_0$. Despite having made this definition for rough differential operators, in this subsection, we shall be devoted to considering their smooth counterparts.
Recall a rough differential operator is smooth if $\Dir:\Ck{\infty}(\cM;\cE) \to \Ck{\infty}(\cM;\cE)$. 
Throughout, we fix a smooth first-order differential operator $\Dir$.

For a LECM-RHM pair $(\mu,\mh)$, we define the two following operators associated to $\Dir$ by specifying their domains:
\begin{align*} 
\dom(\Dir_{\cc}) &:= \Ck[cc]{\infty}(\cM;\cE), \quad \text{and}\\
\dom(\Dir_{\dd}) &:= \set{ u \in \Ck{\infty} \cap \Lp{2}(\cM;\cE, \mu,\mh): \Dir u \in \Lp{2}(\cM;\cE,\mu,\mh) }.
\end{align*}

\begin{proposition}
\label{Prop:DClosable}
For a LECM-RHM pair $(\mu,\mh)$ over $\cE \to \cM$, the operators $\Dir_{\dd}$ as well as $\Dir_{\cc}$ are closable in $\Lp{2}(\cM;\cE,\mu,\mh)$.
\end{proposition}
\begin{proof}
Since $\Dir_{\cc} \subset \Dir_{\dd}$, if $\Dir_{\dd}$ is closable, it follows that $\Dir_{\cc}$ is also closable.
Therefore, we only  need to show  that $\Dir_{\dd}$ is closable. Fix $u_n \in \dom(\Dir_{\dd})$ such that $u_n \to 0$ and $\Dir_{\dd}u_n \to v$.
We show that $v = 0$. First note that by Definition~\ref{Def:RHM}, we can assume that there is a cover of $\cM$ by $(U_\alpha, \mh_\alpha)$ for which $U_\alpha$ is a smooth compact domain (i.e., $\close{U_\alpha}$ is a smooth compact manifold with boundary) and for which $U_\alpha$ also coincides with a bounded cover $(U_\alpha,\psi_\alpha)$ for $\mu$ as in Definition~\ref{Def:LECM}.  Therefore, at each $x \in \cM$, we are guaranteed some smooth domain $U_x$ and we can further take $V_x$ open containing $x$ such that $\close{V_x} \subset U_x$.  Take now  $\eta \in \Ck[c]{\infty}(\cM;[0,1])$ such that $\eta = 1$ on $V_x$ and $0$ on an open set containing $\partial U_x$.
By the local comparability condition, we have that $\eta u_n \in \Lp{2}(U_x;\cE,\mh_x^\infty,\psi_x^\ast (\Leb \restrict \psi_x (U_x)))$ and certainly, 
$$\norm{\eta u_n}_{\Lp{2}(U_x; \cE, \mh^\infty_x, \psi_x^\ast (\Leb \restrict \psi_x (U_x)))} 
\lesssim \norm{\eta u_n}_{\Lp{2}(U_x; \cE, \mh, \mu)} 
\lesssim \norm{u_n}_{\Lp{2}(M; \cE, \mh, \mu)} \to 0$$
as $n \to \infty$.
Moreover, 
$$ \Dir (\eta u_n) = \eta \Dir u_n + [\Dir, \eta \ident] u_n$$
and since $\Dir$ is first-order, we have that $[\Dir, \eta\ident] =\sym_{\Dir}(\extd \eta)$ which is bounded since $\eta$ is compactly supported.   Hence  $[\Dir, \eta \ident] u_n \to 0$ and, therefore, $\Dir (\eta u_n) \to \eta v$ 
as $n \to \infty$ inside $\Lp{2}(U_x; \cE, \mh_x^\infty, \psi_x^\ast (\Leb \restrict \psi_x (U_x)))$.
This shows that  $\eta u_n \in \dom(\Dir_{\min, \mh_x^\infty, \psi_x^\ast (\Leb \restrict \psi_x (U_x))})$, the closed minimal extension of $\Dir$ inside $U_x$ with respect to $(\mh_x^\infty,\psi_x^\ast (\Leb \restrict \psi_x (U_x)))$, which are all smooth objects. 
Since $\eta u_n \to 0$ and $\Dir(\eta u_n) \to \eta v$, by the closedness of $\Dir_{\min, \mh_x^\infty, \psi_x^\ast (\Leb \restrict \psi_x (U_x))}$, we have that $\eta v \in \dom(\Dir_{\min, \mh_x^\infty, \psi_x^\ast (\Leb \restrict \psi_x (U_x))})$ and $\eta v = 0$. 
Therefore $\eta v = 0$ almost-everywhere in $U_x$ and since $\eta = 1$ on $V_x$, we have that $v\rest{V_x} = 0$ almost-everywhere in $V_x$. 
Since $\cM$ can be covered by such $V_x$, we have that $v = 0$ almost-everywhere in $\cM$.
\end{proof} 

By the closability of $\Dir_{\cc}$ and $\Dir_{\dd}$ we have just established, we define their closures in the following definition.

\begin{definition}
\label{Def:D0D2}
Given a LECM-RHM pair $(\mu,\mh)$, define 
$$ \Dir_{0,\mu,\mh} := \close{\Dir_{\cc}}\quad\text{and}\quad \Dir_{2,\mu,\mh} := \close{\Dir_{\dd}}.$$
When the context of $(\mu,\mh)$ is fixed and clear, we simply write $\Dir_0$ and $\Dir_2$.
\end{definition}

Note that by construction, we have that $\Dir_0 \subset \Dir_2$.
In the setting that $(\mu,\mh)$ is smooth, it is immediate that $\Dir_0 = \Dir_{\min}$, where $\Dir_{\min}$ is the minimal extension of the operator $\Dir$. 
In that case, we can also consider the maximal extension $\Dir_{\max} = (\Dir^\dagger_{\cc})^\ast$, where $\Dir^\dagger$ is the formal adjoint with respect to a smooth $\mh$.
We consider a general RHM-LECM pair  $(\mh,\mu)$, we may attempt to extract a replacement for $\Dir^\dagger$ by considering $\Dir_0^\ast$.
This is motivated by the smooth setting where it is immediate that $\Dir^\ast_0\rest{\Ck[cc]{\infty}(\cM;\cE)} = \Dir^\dagger_{\cc}$.
A necessary condition in this statement is that $\Ck[cc]{\infty}(\cM;\cE) \subset \dom(\Dir_0^\ast)$.
It is precisely this containment which may fail when we consider the adjoint $\Dir_0^\ast$ with respect to a general, non-smooth,  RHM-LECM pair $(\mu,\mh)$.
See Subsection~\ref{Sec:NonsmoothDom} for an explicit example. 

In the smooth setting, there are many instances where $\Dir_2 = \Dir_{\max}$. 
For instance, assuming that $\Dir$ is first-order elliptic $\cM$ with compact boundary $\dM$, if $\Dir$ is complete (compactly supported sections with support up to the boundary are dense in the maximal domain), then $\Dir_2 = \Dir_{\max}$.
This is the setup considered in \cite{BaerBan} in their analysis of elliptic first-order problems with compact boundary and in \cite{BGS} for general-order elliptic problems on compact manifolds with boundary.


Moreover, as the following proposition shows, there are potentially many nilpotent extensions of $\Dir_0$ sitting inside $\Dir_2$, induced by their action on $\Ck{\infty}(\cM;\cE)$. 

\begin{proposition} 
\label{Prop:NilExt}
Let $\Dir$ be a smooth coefficient first-order differential operator. 
Suppose that $\Dir_{e,c}$ is an extension (not necessarily closed) arising from a differential operator $\Dir$ with $\dom(\Dir_{e,c}) \subset \Ck{\infty}(\cM;\cE)$ and that $\ran(\Dir_{e,c}) \subset \dom(\Dir_{e,c})$.
If $\Dir^2 = 0$ on $\Ck{\infty}(\cM;\cE)$, then the closure  $\Dir_{e} := \close{\Dir_{e,c}}$ is nilpotent.
\end{proposition}
\begin{proof}
Fix $v := \Dir_e u$.
By definition, there exists $u_n \in \dom(\Dir_{e,c})$ with $u_n \to u$ and $\Dir_{e,c} u_n \to \Dir_{e,c} u$.
Let $v_n := \Dir u_n = \Dir_{e,c} u_n$, since $\dom(\Dir_{e,c}) \subset \Ck{\infty}(\cM;\cE) = \dom(\Dir)$, where $\Dir$ is the associated (rough) differential operator to $\Dir_e$.
Since $\ran(\Dir_{e,c}) \subset \dom(\Dir_{e,c})$, we have that $v_n \in \dom(\Dir_{e,c})$.
Also, $\dom(\Dir_{e,c}) \subset \Ck{\infty}(\cM;\cE)$ and $\Dir$ is smooth coefficient, $v_n \in \Ck{\infty}(\cM;\cE)$.
Therefore, $v_n \to v$ and $0 = \Dir_{e} v_n$ by the fact that $\Dir^2 = 0$ on $\Ck{\infty}(\cM;\cE)$. 
So $\Dir_{e} v_n \to 0$.
Together, since $\Dir_{e}$ is closed, we have that $v \in \dom(\Dir_{e})$ and $\Dir_{e} v = 0$. 
This proves that $\ran(\Dir_e) \subset \ker(\Dir_e)$.
\end{proof}

An important corollary is the following, which shows that the two extreme ends of the scales of operators that play the role of the minimal and maximal extension are nilpotent.
In particular, this shows that there are always nilpotent extensions for smooth coefficient nilpotent differential operators. 
\begin{corollary}
\label{Cor:MaxMinNil}
If $\Dir^2 = 0$, the extensions $\Dir_2$ and $\Dir_0$ are both nilpotent.
\end{corollary}
\begin{proof}
If $u \in \dom(\Dir_{\cc})$, then $u \in \Ck[cc]{\infty}(\cM;\cE)$ and by the locality of $\Dir$, we have that $\Dir u \in \Ck[cc]{\infty}(\cM;\cE)$.
Then the nilpotency of $\Dir_0$ follows. If $u \in \dom(\Dir_{\dd})$, then $u \in \Ck{\infty}(\cM;\cE) \cap \Lp{2}(\cM;\cE,\mu,\mh)$ and $Du \in \Ck{\infty}(\cM;\cE) \cap \Lp{2}(\cM;\cE,\mu,\mh)$.
Letting $v := Du$, we have that $\Dir v = \Dir^2 u = 0 \in \Ck{\infty}(\cM;\cE) \cap \Lp{2}(\cM;\cE,\mu,\mh)$. 
Therefore, $v \in \dom(\Dir_{\dd})$.
This shows that $\ran(\Dir_{\dd}) \subset \dom(\Dir_{\dd})$.
\end{proof}

Combining these, we obtain the following alternative characterisation of nilpotent extensions. 
\begin{corollary}
\label{Cor:NilExt0}
If $\Dir^2=0$, then an extension $\Dir_{e}$ of $\Dir$ is nilpotent if and only if $\ran(\Dir_{e}) \subset \dom(\Dir_{e})$.
\end{corollary}
\begin{proof}
Suppose that $\Dir_{e}$ is nilpotent.
Then $\ran(\Dir_{e}) \subset \ker(\Dir_{e}) \subset \dom(\Dir_{e})$. For the converse statement, note that since $\Dir^2 = 0$, by Corollary~\ref{Cor:MaxMinNil}, we have that $\ran(\Dir_{e}) \subset \ker(\Dir_{2})$.
But since $\ran(\Dir_{e}) \subset \dom(\Dir_{e})$, we have that $\ran(\Dir_{e}) \subset \dom(\Dir_{e}) \cap \ker(\Dir_{2}) = \ker(\Dir_{e})$.
\end{proof}

On manifolds with boundary, we have the following nilpotency result for ``mixed'' boundary conditions.

\begin{corollary} 
Let $\cM$ be a manifold with boundary and $V \subset \dM$ open.
Define 
$$ \dom(\Dir_{V}^0) := \set{u \in \Ck{\infty} \cap \Lp{2}(\cM;\cE,\mu,\mh): \Dir u \in \Lp{2}(\cM;\cE,\mu,\mh),\quad u\rest{V} = 0,\quad (\Dir u)\rest{V} = 0}$$
(where for $V =\emptyset$, we take $\Dir_{\emptyset}^0 = \Dir_{\dd}$).
If $\Dir^2 = 0$ on $\Ck{\infty}(\cM;\cE)$, then $\Dir_{V} := \close{\Dir_{V}^0}$ is nilpotent. 
\end{corollary} 
\begin{proof}
Let $u \in \dom(\Dir_{V}^0)$ so that $u \rest{V} = 0$ and $(\Dir u)\rest{V} = 0$.
Setting $v := \Dir u$, we have that $v \rest{V} = (\Dir u)\rest{V} = 0$ and by the fact that $\dom(\Dir_{V}^0) \subset \Ck{\infty}(\cM;\cE)$, we have that $\Dir v = \Dir^2 u = 0$.
Therefore, $\ran(\Dir_{V}^0) \subset \dom(\Dir_{V}^0)$. 
\end{proof}  

\section{Rough Riemannian Metrics and the Hodge-Dirac operator}
\label{Sec:RRMHodge}

We consider an important application of the general results in Section~\ref{Sec:SetupResults} and Section~\ref{Sec:GH} to a natural geometric setup by considering the case of the the Hodge-Dirac operator in non-smooth situations.
For that, we recall the following definition from \cite{BRough}.

\begin{definition}[Rough Riemannian Metric on $\cM$]
\label{Def:RRM} 
We say that $\mg \in \Lp{0}(\Sym_{\R} \cotanb\cM \tensor \cotanb\cM)$ is a \emph{Rough Riemannian Metric (RRM)}  if for each $x \in \cM$, there exists a chart $(U,\psi)$ with $x \in U$ and a constant $C(U) \geq 1$ such that $$ \frac{1}{C(U)} \modulus{\psi_\ast (y) u}_{\R^n} \leq \modulus{u}_{g(y)} \leq C(U) \modulus{\psi_\ast (y) u}_{\R^n}$$ for $y$-almost-everywhere in $U$ and for all $u \in \tanb_y \cM$.
The structure $(U,\psi)$ is called a local comparability chart. 
\end{definition}

First, we verify that this definition coincides with  $\mg$ being a RHM on $\tanb\cM$ but with real coefficients.
\begin{proposition} 
\label{Prop:RRM} 
The following are equivalent: 
\begin{enumerate}[label=(\roman*)]
\item 
\label{Itm:RRM:1}
$\mg$ is real RHM on $\tanb\cM$.
\item 
\label{Itm:RRM:2} 
$\mg$ is a RRM.
\end{enumerate} 
\end{proposition}
\begin{proof} 
To prove \ref{Itm:RRM:1} implies \ref{Itm:RRM:2}, for a fixed $x \in \cM$, we let $(U_\alpha, \mh_{\alpha}^\infty)$ be the local comparability structure with $x \in U_\alpha$. 
Now, let $(V,\psi)$ be a chart around $x$ and without the loss of generality, assume $V \subset U_\alpha$.
Inside $V$, let $U$ be an open set with $\close{U} \subset V$ and $\close{U}$ compact.
For all $y \in U$, we have a constant $C \geq 1$ with ${C}^{-1} \modulus{\psi_\ast(y) u}_{\R^n} \leq \modulus{u}_{\mh_\alpha^\infty(y)} \leq {C} \modulus{\psi_\ast(y) u}_{\R^n}$ 
for all $y \in U$.
It is clear that \ref{Itm:RRM:2} follows from using the comparability condition in Definition~\ref{Def:RHM}. To prove \ref{Itm:RRM:2} yields \ref{Itm:RRM:1}, we simply note that for each local comparable chart $(U,\psi)$, the metric $\mh_{U}^\infty = \psi^\ast \inprod{\cdot,\cdot}_{\R^n}$, is smooth. 
Therefore, $(U,\mh_{U}^\infty)$ is the local comparability structure and this shows that $\mg$ is real RHM. 
\end{proof}

Moreover, inside two overlapping locally comparable charts $(U,\phi)$ and $(V,\psi)$, it is readily verified that
$$ \int_{\phi(A)} \sqrt{\det \mg \comp \phi^{-1}(x})\ d\Leb(x) = \int_{\psi(A)} \sqrt{\det \mg \comp \psi^{-1}(x})\ d\Leb(x)$$ 
whenever $A \subset U \cap V$ is measurable (c.f. Proposition~4 in \cite{BRough}).
This allows us to consistently extend the usual definition of the volume measure for the smooth case to the rough case as follows.

\begin{definition}[Induced volume measure]
\label{Def:VolMeas}
If $\mg$ is an RRM on $\cM$, then writing
\[ 
d\mu_{\mg}(x) = \sqrt{\det \mg(x)}\ \psi^{\ast}\Leb(x)
\]
inside a local comparability chart $(U,\psi)$ with $\psi^{\ast}\Leb$ is the pullback of the Lebesgue measure in $U$, is the \emph{volume measure} induced by $\mg$.
\end{definition}

In order to ensure we can apply the general results we have obtained thus far, we prove the following.

\begin{proposition} 
\label{Prop:VolMeas}
The induced volume measure $\mu_{\mg}$ from a RRM $\mg$ is a LECM.
\end{proposition}
\begin{proof}
Proposition 5 in \cite{BRough} provides us with $\Meas(\mu_{\mg}) = \Meas$ and $\Zero(\mu_{\mg}) = \Zero$ and Proposition 6 in \cite{BRough} yields $\mu_{\mg}$ is Borel and finite on compact sets.
Define the functional $L_{\mg}: \Ck[c]{0}(\cM) \to \R$ by 
$$ L_{\mg}(f) = \int_{\cM} f\ d\mu_{\mg}.$$
Clearly, this is a functional since $\mu_{\mg}$ is Borel and finite on compact sets.
Fix a compact $K \subset \cM$ and let $f \in \Ck[c]{0}(\cM)$ with $\modulus{f} \leq 1$ and $\spt f \subset K$.
Then, 
$$ L_{\mg}(f) \leq \modulus{L_{\mg}(f)} = \modulus{\int_{K} f\ d\mu_{\mg}} \leq \int_{K} \modulus{f}\ d\mu \leq \mu_{\mg}(K).$$
That is, 
$$ \sup \set{L_{\mg}(f): f \in \Ck[c]{0}(\cM),\ \modulus{f} \leq 1, \spt f \subset K} \leq \mu_{\mg}(K) < \infty.$$
This allows us to invoke the Riesz Representation Theorem, Theorem 4.1 in Chapter 1 of \cite{Simon} to obtain a Radon measure $\mu$ such that 
$$ L_{\mg}(f) = \int_{\cM} f\ d\mu$$
for all $f \in \Ck[c]{0}(\cM)$ since $L_{\mg}(f) \geq 0$ for $f \geq 0$ (c.f. Remark 4.3 in Chapter 1 of \cite{Simon}).
By the density of $\Ck[c]{0}(\cM)$ in $\Lp{p}(\cM;\mu)$, we have that $\mu = \mu_{\mg}$.
\end{proof} 

\begin{corollary} 
Denoting the induced metric from $\mg$ on the complexified bundle $\Forms \cM \to \cM$ again by $\mg$, we have that $(\mg,\mu_{\mg})$ is an RHM-LECM pair.
\end{corollary}

Recall $\extd: \Ck{\infty}(\Forms[k] \cM) \to \Ck{\infty}(\Forms[k+1]\cM)$, the exterior derivative on $k$ forms.
Furthermore, $\extd: \Ck{\infty}(\Forms[k] \cM) \to \Ck{\infty}(\Forms[k+1]\cM)$ and it satisfies $\extd^2 = 0$ on $\Ck{\infty}(\Forms\cM)$.
Moreover, the a RRM $\mg$ canonically extends, to the complex bundle $\Forms[k] \cM$ and hence $\Forms\cM$.
Therefore, we have the following theorem as a consequence of Theorem~\ref{Thm:GenHodgeMain}.

\begin{theorem}[Hodge-type theorem for the Hodge-Dirac operator]
\label{Thm:GenHodgeDiracMain}
If $\mg_1$ and $\mg_2$ are two rough metrics that are mutually bounded, then the pairs $(\mu_{\mg_1}, \mg_1)$ and $(\mu_{\mg_2},\mg_2)$ are also mutually bounded  and yield equivalent $\Lp{2}$-spaces which are equal as sets.
Let  $\extd_e$ be a closed and nilpotent extension of $\extd_{\cc}$, which remains fixed with respect to $(\mu_{\mg_i}, \mg_i)$.
Define  the Hodge-Dirac operator with respect to the metric $\mg_i$ by $\DDir_{e,i,\Hodge} := \extd_e + \extd_e^{\ast, \mh_i}$.
Then: 
\begin{enumerate}[label=(\roman*)]
\item 
\label{Itm:GenHodgeDiracMain:1}
The operator $\DDir_{e,i,\Hodge}$ is self-adjoint with respect to $\inprod{\cdot,\cdot}_{(\mg_i)}$ and in particular, it is densely-defined and closed. 
More generally, $\DDir_{e,i,\Hodge}$ is bi-sectorial with respect to either norm.
\item 
\label{Itm:GenHodgeDiracMain:2}
For all $k,l,k',l' \in \Na\setminus\set{0}$, it holds that $\ker(\DDir_{e,i,\Hodge}^k) = \ker(\DDir_{e,i,\Hodge})$ and $\close{\ran(\DDir_{e,i,\Hodge}^l}= \close{\ran(\DDir_{e,i,\Hodge})}$, and moreover,
\[
\Lp{2}(\cM;\cE) 
= \ker(\DDir_{e,1,\Hodge}^k) \oplus \close{\ran(\DDir_{e,1,\Hodge}^l)} 
= \ker(\DDir_{e,2,\Hodge}^{k'}) \oplus \close{\ran(\DDir_{e,1,\Hodge}^{l'})},
\]
where the splitting is not in general orthogonal.
\item 
\label{Itm:GenHodgeDiracMain:3}
We have that  $\ker(\DDir_{e,1,\Hodge}) \cong \ker(\DDir_{e,2,\Hodge})$ where the isomorphism is given by 
$$ \Phi :=  \Proj{\ker(\DDir_{e,1,\Hodge}), \close{\ran(\DDir_{e,1,\Hodge})}}\rest{\ker(\DDir_{e,2,\Hodge})}: \ker(\DDir_{e,2,\Hodge}) \to \ker(\DDir_{e,1,\Hodge}).$$
The inverse is then given by
$$ \Phi^{-1} :=  \Proj{\ker(\DDir_{e,2,\Hodge}), \close{\ran(\DDir_{e,1,\Hodge})}}\rest{\ker(\DDir_{e,1,\Hodge})}: \ker(\DDir_{e,1,\Hodge}) \to \ker(\DDir_{e,2,\Hodge}).$$
\item 
\label{Itm:GenHodgeDiracMain:4}
We have that $\ker(\DDir_{e,1,\Hodge}\rest{\Forms[k]\cM}) \cong \ker(\DDir_{e,2,\Hodge}\rest{\Forms[k]\cM})$, where the isomorphisms is given by
$$ \Phi_k :=  \Proj{\ker(\DDir_{e,1,\Hodge}), \close{\ran(\DDir_{e,1,\Hodge})}}\rest{\ker(\DDir_{e,2,\Hodge}\rest{\Forms[k]\cM})}: \ker(\DDir_{e,2,\Hodge}\rest{\Forms[k]\cM} ) \to \ker(\DDir_{e,1,\Hodge}\rest{\Forms[k]\cM} ).$$
\end{enumerate}
\end{theorem}
\begin{proof}
For $\mg_1$ and $\mg_2$ rough, we have an $A \in \Lp{0}(\cM;\Sym_{\R} \cotanb \cM \tensor \tanb \cM)$ such that $\mg_1(x)[u,v] = \mg_2(x)[A(x)u,v]$ by invoking Proposition~\ref{Prop:RRM} followed by Proposition~\ref{Prop:MutMat}.
Since $\mg_1 \sim \mg_2$, Proposition~\ref{Prop:MutMat} further yields $A, A^{-1} \in \Lp{\infty}(\cM;\cotanb\cM \tensor \tanb\cM, \mh_i)$. 
If the mutual boundedness constant is $C \geq 1$ for $\mu_1 \sim \mu_2$, the measures then satisfy $d\mu_{\mg_1}(x) = \sqrt{\det A(x)}\ d\mu_{\mg_2}(x)$ with 
$$ \frac{1}{C^{\frac2 n}} \leq \sqrt{\det A} = \frac{d\mu_{\mg_1}}{d\mu_{\mg_2}} \leq {C^{\frac2 n}}.$$
Clearly, as the metrics $\mg_i$ extend to the exterior bundle, so does the endomorphisms $A$ with bounds again dependent on $C$. 
Therefore, $\mu_{\mg_1} \sim \mu_{\mg_2}$ so that $(\mu_{\mg_1},\mg_1) \sim (\mu_{\mg_2},\mg_2)$.
Setting $\Dir_{e} = \extd_e$, we invoke Theorem~\ref{Thm:GenHodgeMain} and we obtain the desired conclusions.  
\end{proof}

As an immediate consequence, we obtain the following corollary since $\extd$ is a smooth coefficient first-order differential operator. 

\begin{corollary} 
\label{Cor:HodgeDiracThmMinMaxExt}
The conclusions of Theorem~\ref{Thm:GenHodgeMain} hold for the extensions $\extd_0$ and $\extd_2$, with the Hodge-Dirac operators being $\DDir_{0, i} := \extd_0 + \extd^{0,\mg_i}$, and $\DDir_{2, i} := \extd_2 + \extd^{2,\mg_i}$ respectively. 
\end{corollary}
\begin{proof}
Clearly, $\extd$ is a smooth coefficient first-order differential operator. 
Then, by Corollary~\ref{Cor:MaxMinNil}, we obtain that $\extd_0$ and $\extd_2$ are both nilpotent.
The conclusions of the corollary follow simply from Theorem~\ref{Thm:GenHodgeMain}.
\end{proof} 

When the underlying manifold $\cM$ is compact and without boundary, we have the following important generalisation of the Hodge-theorem.
Recall that the $k$-th (smooth) de-Rham cohomology is given by 
$$\Hom[k+1]_{\dR}(\cM) := \faktor{\ker(\extd\rest{\Ck{\infty}(\cM; \Forms[k+1]\cM)})} {\ran(\extd\rest{\Ck{\infty}(\cM;\Forms[k]\cM)})}.$$

\begin{theorem}[Hodge-Dirac theorem for rough metrics]
\label{Thm:HodgeDiracThm}
Let $\mg$ be a RRMs on a compact manifold $\cM$ with $\dM = \emptyset$.
Then, there is a unique extension $\eextd = \extd_0$, the Hodge-Dirac operator $\DDir_{\mg, \Hodge} := \eextd + \eextd^{\ast,\mg}$ is self-adjoint and $\Hom[k]_{\dR}(\cM) \cong  \ker(\DDir_{\mg,\Hodge}\rest{\Forms[k]\cM}).$
\end{theorem}

Similar assertions hold for Hodge-Dirac operators $\DDir_{e,\mg} = \extd_e  + \extd_e^{\ast,\mg}$ arising from nilpotent extensions on non-compact manifolds. 
However, in this case, the smooth cohomology $\Hom[k]_{\dR}(\cM)$ needs to be replaced with the $\Lp{2}$-cohomology  $\displaystyle{\Hom[k]_{\dR}(\extd_{e}) = \faktor{\ker \extd_{e}\rest{\Forms[k]\cM} }{\close{\ran(\extd_{e}\rest{\Forms[k+1]\cM}})}}$. 
Seen from this point of view, the passage to the smooth cohomology as we have written in Theorem~\ref{Thm:HodgeDiracThm} for the compact case can be obtained by the uniqueness of the extension coupled with elliptic regularity on changing to a smooth auxiliary background, reducing the question to smooth results.

\appendix
\section{Abstract Hodge-Dirac-type theory}
\label{Appendix} 

\subsection{Nilpotent operators, Hodge decompositions and cohomology}

Let $\Hil$ be a separable Hilbert space and $\Gamma: \dom(\Gamma) \to \Hil$ a closed, densely-defined, and nilpotent operator. 
Here, nilpotency means that $\ran(\Gamma) \subset \ker(\Gamma)$.
Since $\Gamma$ is closed, we have that $\ker(\Gamma)$ is closed, and the nilpotent condition self improves to $\close{\ran(\Gamma)} \subset \ker(\Gamma)$.

All of the results presented here until the end of this subsection are proved in \cite{AKMc} (in fact, in even greater generality). 
They can also be found in \cite{AxMc} as well as in Chapter~10.1 in \cite{RosenBook}. 
However, we include it here in order to provide complete proofs of the results we obtain in this paper. 

\begin{lemma}
\label{Lem:AbsAdj}
The operator $\Gamma^\ast: \dom(\Gamma^\ast) \subset \Hil \to \Hil$ is a closed, densely-defined, nilpotent operator.
\end{lemma}
\begin{proof}
The fact that $\Gamma^\ast$ exists as a closed and densely-defined operator follows from operator theory, c.f. Proposition III-5.29 \cite{Kato}.
For nilpotency, let $v \in \ran(\Gamma^\ast)$ and $u \in \dom(\Gamma)$ and so,
$$ \inprod{v, \Gamma u} = \inprod{\Gamma^\ast w, \Gamma u} = \inprod{w, \Gamma^2 u } = 0.$$
This shows that $v \in \dom(\Gamma^\ast)$ and that $\Gamma^\ast v = 0$.
That is, $v \in \ker(\Gamma^\ast)$.
\end{proof}

\begin{definition}[Hodge-Dirac-type operator]
Define  $\Pi: \dom(\Pi) \subset \Hil \to \Hil$ with domain $\dom(\Pi) = \dom(\Gamma) \cap \dom(\Gamma^\ast)$ by  $\Pi := \Gamma + \Gamma^\ast$.
\end{definition} 

\begin{lemma}
\label{Lem:PiBd} 
For all $u \in \dom(\Pi)$, we have that
$$  \norm{\Gamma u} + \norm{\Gamma^\ast u} \simeq \norm{\Pi u}.$$
\end{lemma}
\begin{proof}
For $u \in \dom(\Pi)$, we see that
$$
\norm{\Gamma u}^2 = \norm{\Gamma u}^2 + \inprod{(\Gamma^\ast)^2 u, u} = \inprod{\Gamma u, \Gamma u} + \inprod{\Gamma^\ast u, \Gamma u} = \inprod{\Pi u, \Gamma u} \leq \norm{\Pi u} \norm{\Gamma u}.$$
Similarly, we have $\norm{\Gamma^\ast u} \leq \norm{\Pi u}.$ The reverse inequalities are immediate.
\end{proof}

\begin{corollary}
\label{Cor:KerPi}
The kernel $\ker(\Pi) = \ker(\Gamma) \cap \ker(\Gamma^\ast)$.
\end{corollary}
\begin{proof} 
It is easy to see that $\ker(\Gamma) \cap \ker(\Gamma^\ast) \subset \dom(\Gamma)$ by definition and moreover, that $\ker(\Gamma) \cap \ker(\Gamma^\ast) \subset \ker(\Gamma)$.
The reverse containment follows immediately the estimate in Lemma~\ref{Lem:PiBd}
\end{proof} 

\begin{corollary}
\label{Cor:PiClosed}
The operator $\Pi$ is closed. 
\end{corollary}
\begin{proof}
Let $\dom(\Pi) \ni u_n \to u$ and $\Pi u_n \to v$.
From Lemma~\ref{Lem:PiBd}, we have that  
$$\norm{\Gamma u_n - \Gamma u_m} + \norm{\Gamma^\ast u_n - \Gamma^\ast u_m} \lesssim \norm{\Pi u_n - \Pi u_m} \to 0.$$
Therefore, $\Gamma u_n \to w$ and $\Gamma^\ast u_n \to w^\ast$.
By the closedness of $\Gamma$ and $\Gamma^\ast$, we obtain that $u \in \dom(\Gamma)$ and $u \in \dom(\Gamma^\ast)$ and $w = \Gamma u$, $w^\ast = \Gamma^\ast u$.
Therefore, $u \in \dom(\Pi)$ and $\Pi u = v$.
\end{proof}

Now, we obtain the main Hodge-Decomposition with regards to this operator.

\begin{proposition}
\label{Prop:HodgeDecomp'}
We have
$$\Hil = \ker(\Gamma) \cap \ker(\Gamma^\ast) \stackrel{\perp} {\oplus} \close{\ran(\Gamma)} \stackrel{\perp} {\oplus} \close{\ran(\Gamma^\ast)}.$$
\end{proposition}
\begin{proof}
From the fact that $\Gamma$ and subsequently $\Gamma^\ast$ are closed, we have that 
$$ \Hil = \ker(\Gamma) \oplus \close{\ran(\Gamma^\ast)} = \ker(\Gamma^\ast) \oplus \close{\ran(\Gamma)}.$$
Therefore, by choosing: 
$$Z' := \ker(\Gamma),\ Z :=\ker(\Gamma^\ast), Y := \close{\ran(\Gamma)}, Y':= \close{\ran(\Gamma^\ast)},$$
and using the fact that  $Y = \close{\ran(\Gamma)} \subset \ker(\Gamma) = Z'$ and similarly $Y' \subset Z$, Proposition~\ref{Prop:Decomp} yields the desired result.
\end{proof}

By the nilpotency of $\Gamma$, as aforementioned, we have that $\close{\ran(\Gamma)} \subset \ker(\Gamma)$.
Therefore, we can define the $\Gamma$-cohomology as:
$$ \Hom(\Gamma) := \faktor{\ker(\Gamma)}{\close{\ran(\Gamma)}}.$$
An immediate consequence of the Hodge decomposition from above is the following. 

\begin{corollary}
\label{Cor:CohomGamma}
We have that $\ker(\Gamma) =  \ker(\Pi) \oplus \close{\ran(\Gamma)}$ and therefore, 
$$ \ker(\Pi) \cong \Hom(\Gamma)$$ 
where the isomorphism is given by $\Proj{\ker(\Pi), \close{\ran(\Gamma)}}$. 
\end{corollary}
\begin{proof}
By Corollary~\ref{Cor:KerPi}, we can write the Hodge decomposition in Proposition~\ref{Prop:HodgeDecomp} as
$$ \Hil = \ker(\Pi) \stackrel{\perp} {\oplus} \close{\ran(\Gamma)} \stackrel{\perp} {\oplus} \close{\ran(\Gamma^\ast)}.$$
However, since $\Gamma$ is a closed and densely-defined operator, we also have that
$$ \Hil = \ker(\Gamma) \stackrel{\perp}{\oplus} \close{\ran(\Gamma^\ast)}.$$
Therefore, 
$$\ker(\Gamma) =  \close{\ran(\Gamma^\ast)}^\perp =  \ker(\Pi) \stackrel{\perp} {\oplus} \close{\ran(\Gamma)}.$$
The isomorphism is then simply a consequence.
\end{proof} 

\begin{remark}
Note that the operator $\Pi$ here depends on the metric, as it is built out of both $\Gamma$ and $\Gamma^\ast$. 
Nevertheless, the kernel is now described, through this isomorphism, via the cohomology, which is independent of the metric. 
\end{remark}

\subsection{Hodge-Dirac-type operators}

\begin{proposition} 
\label{Prop:DenselyDefined}
The operator $\Pi: \dom(\Pi) \subset \Hil \to \Hil$ is a densely-defined, closed, self-adjoint operator. 
Moreover, 
$$ \Hil = \ker(\Pi) \stackrel{\perp} {\oplus} \close{\ran(\Pi)}$$ 
and 
$$ \ker(\Pi) = \ker(\Gamma) \cap \ker(\Gamma^\ast)\quad\text{and}\quad 
\close{\ran(\Pi)} = \close{\ran(\Gamma)} \stackrel{\perp}{\oplus} \close{\ran(\Gamma^\ast)}.$$ 
\end{proposition} 
\begin{proof}
\begin{enumerate}[label=(\roman*), itemindent=0em,listparindent=0em, parsep=.5\baselineskip,itemsep=.5\baselineskip] 
\item \label{It:Cor:DD:1}
Claim: $\ker(\Pi) = \ker(\Gamma) \cap \ker(\Gamma^\ast)$.

This is already established in Corollary~\ref{Cor:KerPi}.

\item \label{It:Cor:DD:2}
Claim: for $v \in \close{\ran(\Gamma)}$, there exists $v_n \in \dom(\Pi)$ such that $v_n \to v$. Fix $v \in \close{\ran(\Gamma)}$.
Since $\Gamma^\ast$ is densely-defined, we have a sequence $v_n \in \dom(\Gamma^\ast)$ such that $v_n \to v$ in $\Hil$.
However, we have a splitting $\Hil = \ker(\Gamma^\ast) \oplus \close{\ran(\Gamma)}$.
Letting $v'_n := \Proj{\close{\ran(\Gamma)}, \ker(\Gamma^\ast)}v_n$, we have that $v'_n \to \Proj{\close{\ran(\Gamma)}, \ker(\Gamma^\ast)}v = v$.
Moreover, $v'_n = v - \Proj{ \ker(\Gamma^\ast),\close{\ran(\Gamma)}}v_n \in \dom(\Gamma^\ast) \oplus \dom(\Gamma^\ast) = \dom(\Gamma^\ast)$.
That is, 
$$v'_n \in \close{\ran(\Gamma)} \cap \dom(\Gamma^\ast) \subset \ker(\Gamma) \cap \dom(\Gamma^\ast )\subset \dom(\Gamma) \cap \dom(\Gamma^\ast ) \subset \dom(\Pi),$$
where the first set inclusion follows from the nilpotency of $\Gamma$ providing $\close{\ran(\Gamma)} \subset \ker(\Gamma)$ and demonstrates that $\dom(\Pi)$ is dense in $\close{\ran(\Gamma)}$.

\item Claim: for $w \in \close{\ran(\Gamma^\ast)}$, there exists $w_n \in \dom(\Pi)$ such that $w_n \to w$. This from an identical argument upon interchanging $\Gamma$ and $\Gamma^\ast$ in \ref{It:Cor:DD:2}

\item Claim: $\dom(\Pi)$ is dense in $\Hil$. First, couple \ref{It:Cor:DD:1} with Proposition~\ref{Prop:HodgeDecomp} to obtain 
\begin{equation*} 
\label{Eq:NewHodge}
\Hil = \ker(\Pi) \stackrel{\perp} {\oplus} \close{\ran(\Gamma)} \stackrel{\perp} {\oplus} \close{\ran(\Gamma^\ast)}.
\end{equation*}
So for $u \in \Hil$, we have $u = u_0 + u_1 + u_2 \in \ker(\Pi) \oplus \close{\ran(\Gamma)} \oplus \close{\ran(\Gamma^\ast)}$.
By what we have just established, there exists $u_{1,n} \in \dom(\Pi)$ with $u_{1,n} \to u_1$ and $u_{2,n} \in \dom(\Pi)$ with $u_{2,n} \to u_2$.
Define $w_n :=  u_0 + u_{1,n} + u_{2,n}$ and clearly $w_n \in \dom(\Pi)$ since $u_0 \in \ker(\Gamma)$.
Then, 
\begin{align*} 
\norm{w_n - u} 
= \norm{u_0 + u_{1,n} + u_{2,n} - u_0 - u_1 - u_2} 
&=  \norm{u_{1,n} - u_1 + u_{2,n} - u_2}  \\
&\qquad\leq \norm{u_{1,n} - u_1} + \norm{u_{2,n}- u_2} \to 0
\end{align*} 
as $n \to \infty$.
This shows $\dom(\Pi)$ is dense in $\Hil$. 

\item \label{It:Cor:DD:4}
Claim: $\Pi$ is self-adjoint. We have proved that the operator $\Pi$ is already closed in Corollary~\ref{Cor:PiClosed}. 
Clearly, $\Pi$ is symmetric, so it suffices to prove that $\dom(\Pi^\ast) \subset \dom(\Pi)$.
For that, let $u \in \dom(\Pi^\ast)$. 
That means, for every $v \in \dom(\Pi)$, 
\begin{equation} 
\label{Eq:PiAdj}
\modulus{\inprod{u, \Pi v}} \lesssim \norm{v}.
\end{equation}

To show that $u \in \dom(\Pi)$, we need to show that $u \in \dom(\Gamma^\ast)$ and that $u \in \dom(\Gamma)$.
For that, we show that whenever $w \in \dom(\Gamma^\ast)$, $\modulus{\inprod{u, \Gamma^\ast w}} \lesssim \norm{w}$ establishing $u \in \dom(\Gamma)$ and similarly, $\modulus{\inprod{u,\Gamma x}} \lesssim \norm{x}$ establishing that $u \in \dom(\Gamma^\ast)$. Now, let $w \in \dom(\Gamma^\ast)$.
With respect to the Hodge-decomposition, we write $w = w_0 + w_1 + w_2 \in \ker(\Pi) \oplus \close{\ran(\Gamma)} \oplus \close{\ran(\Gamma^\ast)}$.
Note that $w_1 \in \dom(\Gamma^\ast)$ since $\ker(\Pi) \subset \dom(\Gamma^\ast)$ and $\close{\ran(\Gamma^\ast)} \subset \ker(\Gamma^\ast) \subset \dom(\Gamma^\ast)$.
Clearly, $\Gamma^\ast w = \Gamma^\ast w_1$. But note that by nilpotency of $\Gamma$, $w_1 \in \ker(\Gamma) \subset \dom(\Gamma)$ and since $w_1 \in 
\dom(\Gamma^\ast)$ by assumption, we have that $w_1 \in \dom(\Pi)$.
Then,  $\Pi w_1 = \Gamma^\ast w_1 + \Gamma w_1 = \Gamma^\ast w_1$, again since $w_1 \in \close{\ran(\Gamma)} \subset \ker(\Gamma)$.
Therefore on combining these equalities and using Equation~\eqref{Eq:PiAdj} with $v = w_1$, we obtain that 
$$ \modulus{\inprod{u, \Gamma^\ast w}} 
= \modulus{\inprod{u, \Gamma^\ast w_1}} = \modulus{\inprod{u, \Pi w_1}} 
\lesssim \norm{w_1} 
= \norm{\Proj{\close{\ran(\Gamma)}, \ker(\Pi) \oplus \close{\ran(\Gamma^\ast)}} w} 
\lesssim \norm{w}.
$$
This shows that  that $u \in \dom(\Gamma)$. By similar reasoning, interchanging the role of $\Gamma$ in place of $\Gamma^\ast$ and $w_2$ in place of $w_1$, we obtain that $u \in \dom(\Gamma^\ast)$. 
That is, $u \in \dom(\Pi)$ and so this shows $\dom(\Pi^\ast) \subset \dom(\Pi)$.

\item Claim:  $\close{\ran(\Pi)} = \close{\ran(\Gamma^\ast)} \oplus \close{\ran(\Gamma)}$. From the self-adjointness of $\Pi$ that we established in \ref{It:Cor:DD:4}, $\Hil = \ker(\Pi) \stackrel{\perp}{\oplus} \close{\ran(\Pi)}$.
Since $\ker(\Pi)^\perp  = \close{\ran(\Gamma)} \oplus \close{\ran(\Gamma^\ast)}$, the conclusion follows.
\qedhere 
\end{enumerate} 
\end{proof}

\begin{remark}
Note that from \ref{It:Cor:DD:2} in the proof of Proposition~\ref{Prop:DenselyDefined}, we can actually assert that $\dom(\Pi) \cap \close{\ran(\Gamma)}$ is dense in $\close{\ran(\Gamma)}$.
This readily follows from the alternate calculation
$$\close{\ran(\Gamma)} \cap \dom(\Gamma^\ast) \subset \close{\ran(\Gamma)} \cap \dom(\Gamma) \cap \dom(\Gamma^\ast) \subset \ran(\Gamma) \cap \dom(\Pi).$$
Similarly, $\close{\ran(\Gamma^\ast)} \cap \dom(\Pi)$ is dense in $\close{\ran(\Gamma^\ast)}$. 
\end{remark}

\subsection{Changes of metric}

Let $B: \Hil \to \Hil$ be an self-adjoint isomorphism of $\Hil$.
With respect to $B$, define 
$$ \inprod{u,v}_{B} = \inprod{Bu,v}.$$
We can now consider the relationship between $\Pi =  \Gamma + \Gamma^\ast$ where the adjoint is taken with respect to $\inprod{\cdot,\cdot}$ and the operator $\Pi_{B} = \Gamma + \Gamma^{\ast,B}$, where here, the inner product is taken with respect to $\inprod{\cdot,\cdot}_B$. 
Note that this is equivalent to studying a change of metric. Starting with another inner product $\inprod{ \cdot,\cdot}'$ for $\Hil$, by considering this as a Hermitian symmetric form, we obtain from the first and second representation theorems in Chapter IV, \S2 in \cite{Kato}, a bounded self-adjoint isomorphism $B: \Hil \to \Hil$ such that $\inprod{u,v}'  = \inprod{Bu,v}$.

The results here are motivated by similar results in \cite{AKMc}, except for results from Proposition~\ref{Prop:SpaSum} onward. 
These are new results that we establish in this paper and are key to the geometric applications we have presented in main part of this paper.

\begin{proposition}\label{Prop:HilAdjChange} The operator $\Gamma^{\ast,B} = B^{-1} \Gamma^\ast B$  and $\dom(\Gamma^{\ast,B}) = B^{-1} \dom(\Gamma^\ast)$. 
Moreover, $\Pi_{B} = \Gamma + B^{-1} \Gamma^\ast B$.  
\end{proposition}
\begin{proof}
Let $v \in \dom(\Gamma^{\ast, B})$.
Fixing $u \in \dom(\Gamma)$, we have 
$$ \inprod{u, \Gamma^{\ast,B}v}_B =  \inprod{\Gamma u, v}_B =  \inprod{B \Gamma u, v} = \inprod{ \Gamma u, B v}.$$
This shows that $B v \in \dom(\Gamma^\ast)$ and then,
$$
\inprod{u, \Gamma^{\ast,B}v}_B = \inprod{u , \Gamma^\ast B} = \inprod{B B^{-1} u, \Gamma^\ast B v} = \inprod{u, B^{-1} \Gamma^{\ast} B v}_B.$$
This shows that $\dom(\Gamma^{\ast,B}) \subset \dom(B^{-1} \Gamma^\ast B)$. Conversely, taking $v \in \dom(B^{-1} \Gamma^\ast B)$, we run the  calculation above  in reverse to find that $\dom(\Gamma^{\ast,B}) \supset \dom(B^{-1} \Gamma^\ast B)$. With this, and since $B$ is an isomorphism, it follows that $\dom(\Gamma^{\ast,B}) = \dom(\Gamma^{\ast,B}) = B^{-1} \dom(\Gamma^{\ast})$.
\end{proof}

\begin{proposition}
\label{Prop:HilKerRan}
The space $\ran(\Gamma^{\ast,B}) = B^{-1} \ran(\Gamma^{\ast})$ and  $\ker(\Gamma^{\ast,B}) = \ker(\Gamma^{\ast}B)$. 
Therefore, 
$$
\Hil = \ker(\Gamma) \cap \ker(\Gamma^{\ast} B) \oplus \close{\ran(\Gamma)} \oplus B^{-1} \close{\ran(\Gamma^{\ast})}.$$
\end{proposition}
\begin{proof}
Let $u \in \ran(\Gamma^{\ast,B})$.
Therefore, we have a $v \in \dom(\Gamma^{\ast,B})$ with $u = \Gamma^{\ast,B} v= B^{-1} \Gamma^{\ast} Bv = B^{-1} \Gamma^{\ast} w$, where $w = Bv \in \dom(\Gamma^{\ast})$.
This shows $\ran(\Gamma^{\ast,B}) \subset B^{-1} \ran(\Gamma^{\ast})$.
For the opposite inclusion, let $u \in B^{-1} \ran(\Gamma^{\ast})$.
That means there exists $v \in \dom(\Gamma^{\ast})$ such that  
$$u = B^{-1} \Gamma^{\ast} v =  B^{-1} \Gamma^{\ast} B (B^{-1} v) =  \Gamma^{\ast,B} (B^{-1}v).$$
This shows that $u \in \ran(\Gamma^{\ast,B})$. The kernel equality follows from $\ker(\Gamma^{\ast,B}) = \ker(B^{-1} \Gamma^{\ast}B) = \ker(\Gamma^{\ast}B)$. The Hodge-decomposition 
$$\Hil = \ker(\Gamma) \cap \ker(\Gamma^{\ast,B}) \stackrel{\perp,B}{\oplus} \close{\ran(\Gamma)} \stackrel{\perp,B}{\oplus} \close{\ran(\Gamma^{\ast,B})}$$
follows from the abstract theory, c.f. Proposition~\ref{Prop:HodgeDecomp} and the version in the conclusion simply follows by combining the facts we have established in this proof.
In particular, $\ker(\Gamma_{B}) = \ker(\Gamma) \cap \ker(\Gamma^{\ast,B})$ and $\close{\ran(\Gamma_{B})} = \close{\ran(\Gamma)} \oplus \close{\ran(\Gamma^{B,\ast})}.$
Since $B$ and $B^{-1}$ are an isomorphisms, it is clear we have that $\close{B^{-1} \ran(\Gamma^{\ast,B})} = B^{-1} \close{\ran(\Gamma^{\ast})}$.
\end{proof} 

The results from here on are new and lie at the heart of the geometric applications that we have presented earlier in the paper. Let us now consider the following space: 
$$ \Spa := \ker(\Pi_B) + \close{\ran(\Pi)}.$$
\begin{proposition}
\label{Prop:SpaSum}
We have that 
$$ \Spa = \ker(\Gamma) \cap \ker(\Gamma^{\ast} B) \oplus \close{\ran(\Gamma)} \oplus \close{\ran(\Gamma^\ast)} = \ker(\Pi_B) \oplus \close{\ran(\Pi)}.$$
\end{proposition}
\begin{proof}
First, we show that the subspaces $\ker(\Gamma) \cap \ker(\Gamma^{\ast} B)$, $\close{\ran(\Gamma)}$ and $\close{\ran(\Gamma^\ast)}$ are mutually complementary. 
For that, note $\ker(\Gamma) \cap \ker(\Gamma^{\ast} B) \cap  \close{\ran(\Gamma^\ast)}  \subset \ker(\Gamma) \cap   \close{\ran(\Gamma^\ast)} = 0$  and $\close{\ran(\Gamma)} \cap \close{\ran(\Gamma^\ast)} = 0$ from Proposition~\ref{Prop:HodgeDecomp}.
By applying Proposition~\ref{Prop:HodgeDecomp} with $\inprod{\cdot,\cdot}_B$, we obtain that $\ker(\Gamma) \cap \ker(\Gamma^{\ast} B) \cap  \close{\ran(\Gamma)} = 0$. Lastly, we show $\close{\ran(\Pi)} \cap \ker(\Gamma) \cap \ker(\Gamma^\ast B) = 0$.
Let $y = u + v \in \ker(\Gamma) \cap \ker(\Gamma^{\ast} B) \cap ( \close{\ran(\Gamma)} \oplus \close{\ran(\Gamma^\ast)})$.
Since $u \in \close{\ran(\Gamma)} \subset \ker(\Gamma)$, we have that $v = y - u \in \ker(\Gamma)$. 
Moreover, since $v \in \close{\ran(\Gamma^\ast)}$, we obtain that $v \in \ker(\Gamma) \cap \ker(\Gamma^\ast) \cap \close{\ran(\Gamma^\ast)} = 0$ by Proposition~\ref{Prop:HodgeDecomp}.
Therefore $y = u \in \ker(\Gamma) \cap \ker(\Gamma^{\ast,B}) \cap \close{\ran(\Gamma)} = 0$ again by Proposition~\ref{Prop:HodgeDecomp}.
\end{proof} 

Now we show that $\Spa$ is indeed the entire space $\Hil$.

\begin{proposition}
\label{Prop:SpaTot}  
$\Spa = \ker(\Pi_B) \oplus \close{\ran(\Pi)} = \Hil$.
\end{proposition} 
\begin{proof} 
Applying Corollary~\ref{Cor:CohomGamma} to $\Pi_{B}$ with respect to $\inprod{\cdot,\cdot}_{B}$ yields $\ker(\Gamma) = \ker(\Pi_{B}) \oplus \close{\ran(\Gamma)}$.
Therefore, 
\begin{align*} 
\Spa 
= 
\ker(\Pi_B) \oplus \close{\ran(\Pi)} 
= 
\ker(\Pi_{B}) \oplus \close{\ran(\Gamma)} \oplus \close{\ran(\Gamma^{\ast})} 
=
\ker(\Gamma) \oplus \close{\ran(\Gamma^{\ast})}
= 
\Hil,
\end{align*}
where the last equality follows from the splitting of $\Hil$ via $\Gamma$ and its adjoint $\Gamma^{\ast}$ with respect to $\inprod{\cdot,\cdot}$.
\end{proof} 

This immediately leads us to the following Hodge-type theorem in this abstract setting.

\begin{theorem}[Abstract Hodge-type theorem in Hilbert spaces]
\label{Thm:AbsHodgeType}
The space $\ker(\Pi_B) \cong \ker(\Pi)$.
The isomorphism $\Phi: \ker(\Pi) \to \ker(\Pi_B)$ is precisely given by 
$$\Phi := \Proj{\ker(\Pi_B) \oplus \close{\ran(\Pi)}}\rest{\ker(\Pi)}: \ker(\Pi) \to \ker(\Pi_B),$$
where $\Proj{\ker(\Pi_B) \oplus \close{\ran(\Pi)}}$ is the projection to $\ker(\Pi_B)$ along $\close{\ran(\Pi)}$.
The inverse is then given by
$$ \Phi^{-1} = \Proj{\ker(\Pi) \oplus \close{\ran(\Pi)}}\rest{\ker(\Pi_B)}: \ker(\Pi_B) \to \ker(\Pi).$$
\end{theorem}
\begin{proof}
Recall that since $\Pi$ is self-adjoint, we obtain the splitting 
$$ \Hil = \ker(\Pi) \oplus \close{\ran(\Pi)}.$$ 
Moreover, from Proposition~\ref{Prop:SpaSum} and Proposition~\ref{Prop:SpaTot}, we find
$$ \Hil = \ker(\Pi_B) \oplus \close{\ran(\Pi)}.$$
We then simply invoke Lemma~\ref{Lem:SubIsom} with $\cB = \Hil$, $\cB_1' = \ker(\Pi_B)$, $\cB_2 = \close{\ran(\Pi)}$ and $\cB_1' = \ker(\Pi)$.
To compute $\Phi^{-1}$, we repeat the argument with $\cB_1$ and $\cB_1'$ interchanged, and use the uniqueness of the inverse.
\end{proof}

\begin{remark}
\label{Rmk:AbsHodgeType}
We can equally well argue this using Corollary~\ref{Cor:CohomGamma}. 
Here, we have that $\ker(\Pi_{B}) \oplus \close{\ran(\Gamma)} = \ker(\Pi) \oplus \close{\ran(\Gamma)}$, and therefore, the isomorphism would be given by
\[
\Proj{\ker(\Pi_{B}), \close{\ran(\Gamma)}}\rest{\ker(\Pi)}: \ker(\Pi) \to \ker(\Pi_{B}).
\]
Moreover, we see that $\Hom(\Gamma) \cong \ker(\Pi_{B}) \cong \ker(\Pi)$.
Nevertheless, the value of the splitting in Theorem~\ref{Thm:AbsHodgeType} will become apparent in the following section along the projectors arising there will become apparent in the following section.   
\end{remark}

\subsection{Splittings}
Modelling on the fact that differential forms are graded, we study the situation in which the Hilbert space splits and where a power of an operator respects this splitting.
For that, suppose $\Hil = \Hil_0 \oplus \Hil_1$. 
Given an operator $\Xi: \Hil \to \Hil$, it is not necessarily the case that $\Xi\rest{\Hil_i}$ restricts to an operator $\Hil_i \to \Hil_i$.
However, assuming that it does, we obtain requisite properties to study the situation where $\Xi$ is a factional power of $\modulus{\Pi}$ which respects this splitting.
With this, we prove more refined version of Theorem~\ref{Thm:AbsHodgeType} to obtain an isomorphism at the level of $\Hil_i$. 
This is modelled on the Hodge-Dirac operator $\Dir = \extd + \extd^\ast$, which itself does not preserve $k$-forms, but its square, the Hodge-Laplacian $\Dir^2$, does by  mapping $k$-forms to $k$-forms.
In what is to follow, we say that $[\Proj{\Hil_0,\Hil_1}, \Xi] = 0$   if whenever $u \in \dom(\Xi)$ we have   $\Proj{\Hil_0,\Hil_1}u \in \dom(\Xi)$ and $[\Proj{\Hil_0,\Hil_1}, \Xi]u = 0$.

\begin{lemma}
\label{Lem:SplitMap}
Suppose that $[\Proj{\Hil_0,\Hil_1}, \Xi] = 0$.
Then,  $\Xi\rest{\Hil_0}: \dom(\Xi) \cap \Hil_0 \to \ran(\Xi) \cap \Hil_0$.
Moreover, $\ran(\Xi\rest{\Hil_0}) =  \ran(\Xi) \cap \Hil_0$.
Similar statements hold with $\Hil_0$ replaced by $\Hil_1$.
\end{lemma}
\begin{proof} 
Fix $u \in \dom(\Xi) \cap \Hil_0$, so we can write $u = \Proj{\Hil_0,\Hil_1}u$.
Then, 
\[ 
\Xi\rest{\Hil_0}u = \Xi u = \Xi \Proj{\Hil_0,\Hil_1} u =  \Proj{\Hil_0,\Hil_1} \Xi u \in \Hil_0.
\]
Therefore, we see that $\Xi\rest{\Hil_0} u \in \ran(\Xi) \cap \Hil_0$. 
This shows the mapping properties of $\Xi\rest{\Hil_0}$ and establishes that $\ran(\Xi\rest{\Hil_0}) \subset  \ran(\Xi) \cap \Hil_0$. For the reverse inclusion, let $v \in \ran(\Xi) \cap \Hil_0$ so there exists $u \in \dom(\Xi)$ such that $v = \Xi u$.
Then, 
\[ 
v = \Proj{\Hil_0,\Hil_1}v = \Proj{\Hil_0,\Hil_1}\Xi u  = \Xi \Proj{\Hil_0,\Hil_1} u = \Xi\rest{\Hil_0} \Proj{\Hil_0,\Hil_1} u \in \ran(\Xi\rest{\Hil_0}).
\qedhere
\] 
\end{proof} 

To apply this lemma to splittings, we go back to our operator $\Pi$.
Let 
$$ \modulus{\Pi} = \sqrt{\Pi^2} = \sgn(\Pi) \Pi,$$
where $\sgn(\Pi) =  \chi_{(0,\infty)}(\Pi) - \chi_{(-\infty,0]}(\Pi)$.
First, we have the following important lemma. 

\begin{lemma}
\label{Lem:AlphaDecomp}
For any $\alpha > 0$, $\ker(\Pi) = \ker(\modulus{\Pi}^\alpha)$ and $\close{\ran(\Pi)} = \close{\ran(\modulus{\Pi}^\alpha)}$.
\end{lemma}
\begin{proof}
First, we prove the statement regarding the kernel.
For that, write 
$$ \modulus{\Pi}^\alpha = (1 + \Pi^{2k}) f_\alpha(\Pi),$$
where $k > \alpha$ and 
$$ f_\alpha(\zeta) = \frac{\modulus{\zeta}^\alpha}{1 + \zeta^{2k}},$$
which is clearly holomorphic on bisector containing $\R$.  Let $u \in \ker(\modulus{\Pi}^\alpha)$.
Then, since $(1 + \Pi^{2k})$ is an invertible operator, we have that $f_\alpha(\Pi)u = 0$.
Conversely, we obtain that $\ker (f_\alpha(\Pi)) \subset \ker(\modulus{\Pi}^\alpha)$. 
Therefore, $\ker(f_\alpha(\Pi)) = \ker(\modulus{\Pi}^\alpha)$. By functional calculus considerations, we have that $\ker(\modulus{\Pi}) = \ker(\Pi) \subset \ker(f_\alpha(\Pi)) = \ker(\modulus{\Pi}^\alpha)$.
To obtain the reverse conclusion, take $l > 0$ sufficiently large so that $l \alpha > 1$ and then noting that 
$$\modulus{\Pi} = (\modulus{\Pi}^{l \alpha})^{\frac{1}{l \alpha}}$$ 
along with the fact that $\ker(\modulus{\Pi}^{l\alpha}) = \ker(\modulus{\Pi}^\alpha)$ and interchanging $\Pi$ with $\modulus{\Pi}^\alpha$ in the first argument, we obtain that $\ker(\modulus{\Pi}^{\alpha}) =  \ker(\modulus{\Pi}^{l \alpha}) \subset \ker(\modulus{\Pi})$. Note then that $\close{\ran(\modulus{\Pi}^\alpha)} = \ker(\modulus{\Pi}^\alpha)^{\perp} = \ker(\Pi)^{\perp} =  \close{\ran(\Pi)}$.
\end{proof} 

This leads us to this important proposition. 
\begin{proposition} 
\label{Prop:CommSplit}
Suppose that for some $\alpha > 0$, $[\Proj{\Hil_0, \Hil_1}, \modulus{\Pi}^\alpha] = 0$.
Then, 
\[
\ker(\Pi) \cap \Hil_i = \ker(\modulus{\Pi}^\alpha \rest{\Hil_i})\ \quad\text{and}\quad   
\ran(\Pi) \cap \Hil_i = \ran(\modulus{\Pi}^\alpha \rest{\Hil_i})
\]
along with the splitting 
\[
\Hil_i = \ker(\Pi) \cap \Hil_i \oplus \close{\ran(\Pi)} \cap \Hil_i. 
\]
\end{proposition} 
\begin{proof} 
We consider the case $i = 0$, since $i = 1$ is exactly the same with $\Hil_0$ and $\Hil_1$ interchanged in the argument. 
Suppose that $u \in  \ker(\Pi) \cap \Hil_0 = \ker(\modulus{\Pi}^\alpha) \cap \Hil_0$ for any $\alpha >0$ from Lemma~\ref{Lem:AlphaDecomp}. 
Therefore, it readily follows that 
\[ 
0 =  \modulus{\Pi}^\alpha u = \modulus{\Pi}^\alpha \Proj{\Hil_0,\Hil_1} u = \modulus{\Pi}^\alpha \rest{\Hil_0} u.
\] 
This shows that $u \in \ker(\modulus{\Pi}^\alpha\rest{\Hil_0})$.
Conversely, if we take $u \in \ker(\modulus{\Pi}^\alpha\rest{\Hil_0})$, then it follows that $u \in \ker(\modulus{\Pi}^\alpha) \cap \Hil_0$ and by Lemma~\ref{Lem:AlphaDecomp}, we again have the equality $\ker(\modulus{\Pi}^\alpha)  = \ker(\modulus{\Pi})$.
Therefore, we get that $\ker(\Pi) \cap \Hil_i = \ker(\modulus{\Pi}^\alpha \rest{\Hil_i})$ without requiring the commutator condition on the projector. The range result requires the commutator condition and it is established in Lemma~\ref{Lem:SplitMap}. Now, for the splitting of the space, it is clear that the containment ``$\supset$'' holds. 
Conversely, let $u \in \Hil_0 \subset \Hil$.
By Proposition~\ref{Prop:HodgeDecomp} and Proposition~\ref{Prop:DenselyDefined}, $u = v + w \in (\ker(\Pi) \oplus \close{\ran(\Pi)}) \cap \Hil_0$.
Since $u \in \Hil_0$, $u = \Proj{\Hil_0,\Hil_1} u =  \Proj{\Hil_0,\Hil_1} v + \Proj{\Hil_0,\Hil_1} w$. 
To prove the inclusion ``$\subset$'', it suffices to prove that $\Proj{\Hil_0,\Hil_1} v \in \ker(\Pi)$ and  $\Proj{\Hil_0,\Hil_1} w \in \close{\ran(\Pi)}$. Using the fact that $\ker(\Pi) = \ker(\modulus{\Pi}^\alpha)$ from Lemma~\ref{Lem:AlphaDecomp}, we have that 
\[ 
0 = \Pi v =  \modulus{\Pi}^\alpha v = \Proj{\Hil_0,\Hil_1} \modulus{\Pi}^\alpha v =  \modulus{\Pi}^\alpha \Proj{\Hil_0,\Hil_1}  v.
\]
That is, $\Proj{\Hil_0,\Hil_1}  v \in \ker(\modulus{\Pi}^\alpha) = \ker(\Pi)$. Similarly, $w \in \close{\ran(\Pi)} = \close{\ran(\modulus{\Pi}^\alpha)}$ by Lemma~\ref{Lem:AlphaDecomp}.
Therefore, there exists $w'_n \in \dom(\modulus{\Pi}^\alpha)$ such that $w = \lim_{n\to\infty} \modulus{\Pi}^\alpha w'_n$.
Then, 
\[ 
\Proj{\Hil_0,\Hil_1} w 
= 
\Proj{\Hil_0,\Hil_1} \lim_{n\to\infty} \modulus{\Pi}^\alpha w'_n
=
\lim_{n\to\infty}  \Proj{\Hil_0,\Hil_1}  \modulus{\Pi}^\alpha w'_n
= 
\lim_{n \to \infty} \modulus{\Pi}^\alpha \Proj{\Hil_0,\Hil_1} w'_n
\in \close{\ran(\modulus{\Pi}^\alpha)},
\]
where we use the boundedness of the projector in the second equality and $[\Proj{\Hil_0, \Hil_1}, \modulus{\Pi}^\alpha] = 0$ in the third.
Since $\close{\ran(\Pi)} = \close{\ran(\modulus{\Pi}^\alpha)}$ as we have already mentioned, we have that $\Proj{\Hil_0,\Hil_1} w \in \close{\ran(\Pi)}$. 
\end{proof}

\begin{theorem}[Abstract Hodge-type theorem along a splitting]
\label{Thm:AbsHodgeTypeSplitting} 
Let $\Hil = \Hil_0 \oplus \Hil_1$ and $\Pi = \Gamma + \Gamma^\ast$ and $\Pi_B := \Gamma + B^{-1} \Gamma^\ast B$. 
Suppose  there exists $\alpha > 0$ such that the commutators $[\Proj{\Hil_0,\Hil_1}, \modulus{\Pi}^\alpha] = 0$ and $[\Proj{\Hil_0,\Hil_1}, \modulus{\Pi_B}^\alpha] = 0$ on $\dom(\modulus{\Pi}^\alpha)$ and $\dom(\modulus{\Pi_B}^\alpha)$ respectively. 
Then, 
$$\Hil_0 
= 
\ker(\modulus{\Pi}^\alpha\rest{\Hil_0}) \oplus \close{\ran(\modulus{\Pi}^\alpha\rest{\Hil_0})} 
= 
\ker(\modulus{\Pi_B}^\alpha\rest{\Hil_0}) \oplus \close{\ran(\modulus{\Pi}^\alpha\rest{\Hil_0})}$$
and
$\ker(\Pi\rest{\Hil_0}) \cong \ker(\Pi_B\rest{\Hil_0})$ with the isomorphism given by 
\[ 
\Proj{\ker(\modulus{\Pi_B}^\alpha \rest{\Hil_0}), \close{\ran(\modulus{\Pi}^\alpha\rest{\Hil_0})}}\rest{\ker(\modulus{\Pi}^\alpha\rest{\Hil_0})} 
= 
\Proj{\ker(\Pi_B), \close{\ran(\Pi)}}\rest{\ker(\Pi) \cap \Hil_0}.
\] 
\end{theorem}
\begin{proof}
First we show  $\Hil_0 = \ker(\modulus{\Pi_B}^\alpha\rest{\Hil_0}) \oplus \close{\ran(\modulus{\Pi}^\alpha\rest{\Hil_0})}$.
Fix $v \in \Hil_0$ and  by Proposition~\ref{Prop:SpaTot}, we have that $v = v_0 + v_1 \in \ker( \Pi_B) \oplus \close{\ran(\Pi)}$.
Again, $v = \Proj{\Hil_0, \Hil_1} v = \Proj{\Hil_0, \Hil_1} v_0 + \Proj{\Hil_0, \Hil_1} v_1$. 
Arguing as in Proposition~\ref{Prop:CommSplit}, using both commutator assumptions, we obtain that $\Proj{\Hil_0, \Hil_1} u_1 \in \close{\ran(\modulus{\Pi}^\alpha\rest{\Hil_0})}$ and  $\Proj{\Hil_0, \Hil_1} v_1\in \close{\ran(\modulus{\Pi}^\alpha\rest{\Hil_0})}$.
This shows that $\Hil_0 \subset \ker(\modulus{\Pi_B}^\alpha\rest{\Hil_0}) \oplus \close{\ran(\modulus{\Pi}^\alpha\rest{\Hil_0})}$. 
The opposite containment is immediate. The  splitting $\Hil_0 = \ker(\modulus{\Pi}^\alpha\rest{\Hil_0}) \oplus \close{\ran(\modulus{\Pi}^\alpha\rest{\Hil_0})}$  follows from Proposition~\ref{Prop:CommSplit}, applied to $\Pi$. 
Arguing as in Theorem~\ref{Thm:AbsHodgeType}, using Lemma~\ref{Lem:SubIsom}, we obtain the isomorphism $\ker(\Pi\rest{\Hil_0}) \cong \ker(\Pi_B\rest{\Hil_0})$ via $\Proj{\ker(\modulus{\Pi_B}^\alpha \rest{\Hil_0}), \close{\ran(\modulus{\Pi}^\alpha\rest{\Hil_0})}}\rest{\ker(\modulus{\Pi}^\alpha\rest{\Hil_0})}.$ Furthermore, from this proposition, we have
\begin{align*}
&\ker(\modulus{\Pi}^\alpha\rest{\Hil_0})  = \ker(\Pi) \cap \Hil_0 \\ 
&\ker(\modulus{\Pi_B}^\alpha\rest{\Hil_0})  = \ker(\Pi_B) \cap \Hil_0\\
&\close{\ran(\modulus{\Pi}^\alpha\rest{\Hil_0})} = \close{\ran(\Pi)}\cap \Hil_0, 
\end{align*}
which shows that 
\[
\Proj{\ker(\modulus{\Pi_B}^\alpha\rest{\Hil_0}), \close{\ran(\modulus{\Pi}^\alpha\rest{\Hil_0})}} 
= 
\Proj{\ker(\Pi_B) \cap \Hil_0, \close{\ran(\Pi)}\cap \Hil_0}
= 
\Proj{\ker(\Pi_B), \close{\ran(\Pi)}}\rest{\ker(\Pi) \cap \Hil_0}.
\]
Clearly, the equality $\Proj{\ker(\modulus{\Pi_B}^\alpha \rest{\Hil_0}), \close{\ran(\modulus{\Pi}^\alpha\rest{\Hil_0})}}\rest{\ker(\modulus{\Pi}^\alpha\rest{\Hil_0})}  = \Proj{\ker(\Pi_B), \close{\ran(\Pi)}}\rest{\ker(\Pi) \cap \Hil_0}$ follows from this. 
\end{proof}

\begin{remark}
Note that it is not possible to apply Theorem~\ref{Thm:AbsHodgeType} directly to the operators $\modulus{\Pi}^\alpha$ and $\modulus{\Pi_B}^\alpha$ since these operators cannot be written as the sum of two nilpotent operators. 
\end{remark}

\begin{corollary}
\label{Cor:AbsHodgeTypeSplittingInt}
If $k := \alpha \in \Na$, then 
$$\Hil_0 = \ker(\Pi\rest{\Hil_0}) \oplus \close{\ran({\Pi}^k)} 
= \ker(\Pi_B\rest{\Hil_0}) \oplus \close{\ran({\Pi}^k)}$$
and
$\ker(\Pi\rest{\Hil_0}) \cong \ker(\Pi_B\rest{\Hil_0})$ with the isomorphism given by $\Proj{\ker(\Pi_B\rest{\Hil_0}), \close{\ran({\Pi}^k\rest{\Hil_0})}}$.
\end{corollary}
\begin{proof}
This simply follows from the fact that $\ker(\modulus{\Pi}^\alpha) = \ker(\Pi^k)$, $\ker(\modulus{\Pi_B}^\alpha) = \ker(\Pi_B^k)$ and $\close{\ran(\modulus{\Pi}^k)}  = \close{\ran({\Pi}^k)}$ for $k \in \Na$.
\end{proof}

\subsection{Facts about Banach space isomorphisms}

\begin{proposition} 
\label{Prop:Decomp}
Let $\Spa$ be a Banach space and let $Y, Y', Z, Z'$ be closed subspaces such that 
$$ \Spa = Y \oplus Z = Y' \oplus Z'$$ 
and $Y \subset Z'$ and $Y' \subset Z$.
Then, 
$$ \Spa = Z \cap Z' \oplus Y \oplus Y'.$$
\end{proposition}
\begin{proof}
Let $\Proj{A, B}$ denote the projection to $A$ along $B$ and fix $x \in \Spa$.
Then,
\begin{align*} 
x &= \Proj{Y, Z} x + \Proj{Z,Y} x  \\
&= \Proj{Y',Z'} \Proj{Y,Z}x + \Proj{Z',Y'}  \Proj{Y,Z} x + \Proj{Y',Z'} \Proj{Z,Y}x + \Proj{Z',Y'} \Proj{Z,Y}x \\ 
&= \Proj{Y',Z'} \Proj{Y,Z}x + \Proj{Y',Z'} \Proj{Z,Y}x + \Proj{Z',Y'}  \Proj{Y,Z} x + \Proj{Z',Y'} \Proj{Z,Y}x \\
&= \Proj{Y',Z'} x + \Proj{Y,Z}x + \Proj{Z',Y'} \Proj{Z,Y}x,
\end{align*}
where the ultimate equality follows from the fact that $Y \subset Z'$ from which $\Proj{Z',Y'}  \Proj{Y,Z} x = \Proj{Y,Z} x$.
By the same calculation, interchanging the roles of $Y, Z$ with $Y'$ and $Z'$ and relying on $Y' \subset Z$, we obtain 
$$x = \Proj{Y',Z'} x + \Proj{Y,Z}x + \Proj{Z,Y} \Proj{Z',Y'}x.$$
Therefore,
$$ \Proj{Y',Z'} x + \Proj{Y,Z}x + \Proj{Z,Y} \Proj{Z',Y'}x = \Proj{Y',Z'} x + \Proj{Y,Z}x + \Proj{Z',Y'} \Proj{Z,Y}x$$
from which follows
$$ \Proj{Z,Y} \Proj{Z',Y'}x =  \Proj{Z',Y'} \Proj{Z,Y}x.$$
That is, 
$$ \Proj{Z,Y} \Proj{Z',Y'}x =  \Proj{Z',Y'} \Proj{Z,Y}x \in Z \cap Z'.$$
Therefore, we have shown that 
$$x \in Y + Y' + Z \cap Z'.$$
Now from hypothesis $Y \cap Z = 0$, we have $Y \cap Z \cap Z' = 0, Y' \cap Z \cap Z' = 0$ and $Y \cap Y' \subset Y \cap Z = 0$. Therefore, we obtain that 
$$x \in Y \oplus Y' \oplus Z \cap Z'.$$
Next, we show  that $(Y \oplus Y') \cap Z \cap Z' = 0$.
Let $z := y + y' \in Y \oplus Y' \cap Z \cap Z'$.
Clearly $z = \Proj{Z,Y}z = \Proj{Z,Y} y' = y'$ since $Y' \subset Z$.
On the other hand, $z = \Proj{Z',Y'}z = \Proj{Z',Y'}y = y$ since $Y \subset Z'$.
Therefore, we have that $z = z + z$ and hence $z = 0$.
Together, this shows that 
$$ \Spa \subset Y \oplus Y' \oplus Z \cap Z'.$$
The opposite containment is immediate. 
\end{proof} 

\begin{lemma}
\label{Lem:BsplitIso} 
Let $X = Y \oplus Z =  Y' \oplus Z'$ and suppose that $\Phi: Z \to Z'$ is a Banach space isomorphism.
Then, $X = Y \oplus Z \cong Y  \eoplus Z'$, in the sense of Banach spaces and  where $\eoplus$ denotes the external sum.
\end{lemma}
\begin{proof}
Let $x \in X = Y \oplus Z$. Then, $x = y + z$, for $y \in Y$ and $z \in Z$. We write
\begin{align*}
\norm{x} \simeq \norm{y} + \norm{z}
	= \norm{y} + \norm{ \Phi^{-1} \circ \Phi z} 
	\simeq \norm{y} + \norm{\Phi z} 
\end{align*}
This shows that $x = y + z \mapsto (y, \Phi z)$ is a bounded linear map.
Easily, it is seen to be injective.
To see that it is surjective, let $(y, z') \in Y \oplus Z'$, then $x = y + \Phi^{-1}z \in X$ with $x \mapsto (y, \Phi \Phi^{-1}z) = (y, z')$.
This shows this is a bounded surjection and by the open mapping theorem, an open map.
This along with injectivity means that $x = y + z \mapsto (y, \Phi z)$ is a Banach space isomorphism.
\end{proof}

The following is taken from \cite{BaerBan}. 

\begin{lemma}
\label{Lem:SubIsom}
Let $\cB$ be a Hilbert space such that $\cB = \cB_1 \oplus \cB_2$ where $\cB_1,\ \cB_2 \subset \cB$ are closed subspaces.
Let $\cB_1'$ be another closed subspace such that $\cB = \cB_1' \oplus \cB_2$.
Then, $\Proj{\cB_1, \cB_2}\rest{\cB_1'}: \cB_1' \to \cB_1$ is an isomorphism with bounded inverse.
\end{lemma} 
\begin{proof}
Let $u \in \cB_1$ and note that 
$$ u = \Proj{\cB_1, \cB_2} u = \Proj{\cB_1,\cB_2} (\Proj{\cB_1', \cB_2} u + \Proj{\cB_2, \cB_1'}) u = \Proj{\cB_1,\cB_2} \Proj{\cB_1', \cB_2} u, $$
because $\Proj{\cB_1,\cB_2} \circ \Proj{\cB_2, \cB_1'} = 0$.
This shows surjectivity. Now, let $u \in \cB_1'$ and suppose that $\Proj{\cB_1, \cB_2} u = 0$. 
But then, $u \in \nul(\Proj{\cB_1, \cB_2}) = \cB_2$ and so, $u \in \cB_1' \intersect \cB_2 = \set{0}$, which shows that the map is injective. Since it is a bounded map on all of $\cB_1$, it is bounded, and by the closedness of $\cB_1$ and the open mapping theorem, we can conclude that it has a bounded inverse.
\end{proof}

\bibliographystyle{alpha}

\providecommand{\noopsort}[1]{}\def\cprime{$'$}

\setlength{\parskip}{-.09\baselineskip}
\end{document}